\RequirePackage[log]{snapshot}
\documentclass[11pt,a4paper]{article}
\usepackage[english]{babel}
\usepackage[T1]{fontenc}
\usepackage{ae,aecompl}
\usepackage[utf8]{inputenc} 
 \usepackage{amsthm,amsmath,amsfonts}
\usepackage{comment} 
\usepackage{cite} 
\usepackage[linktocpage]{hyperref} 
\usepackage{graphicx} 
\usepackage{float}
\usepackage{fullpage}
\usepackage{color} 
\usepackage{algorithm}
\usepackage{ifthen}
\usepackage{pdfpages}
\numberwithin{equation}{section} 

\newtheorem{thm}{Theorem}[section]

\newtheorem{prop}[thm]{Proposition}
\newtheorem{lem}[thm]{Lemma}
\newtheorem{rmq}[thm]{Remark}
\newtheorem{rmqs}[thm]{Remarks}
\newtheorem{deff}[thm]{Definition}

\DeclareMathOperator{\sgn}{sgn}
\DeclareMathOperator{\diverg}{div}
\DeclareMathOperator{\Lip}{Lip}

\def\H{\mathcal{H}}

\def\E{\mathcal{E}}

\def\R{\mathbb{R}}

\def\N{\mathbb{N}}

\def\S{\mathbb{S}}
\def\e{\varepsilon}

\def\O{\Omega}
\def\BV{BV(\Omega)}
\def\la{\lambda}
\def\vphi{\varphi}
\def\G{\nabla}

\title{Discontinuities of the minimizers of the weighted or
anisotropic total variation for image reconstruction}

\author{
Khalid Jalalzai\thanks{CMAP, CNRS UMR 7641, \'Ecole Polytechnique, 91128 Palaiseau Cedex, France ({\tt khalid.jalalzai@polytechnique.edu}).}
}

\begin{document}
\date{April 10, 2012}
\maketitle

\begin{abstract}
{The study of the regularity of the minimizer of the weighted anisotropic total variation with a general fidelity term is at the heart of this paper. We generalized some recent results on the inclusion of the discontinuities of the minimizer of the image denoising problem. In particular, we proved that for well-chosen weights and anisotropies, it is actually possible to create discontinuities that were not contained in the original image. We also observed a reduced jump property at the discontinuities of the minimizer. To prove these results we used some regularity theorems for minimal surfaces that we had to adapt to our setting. We also illustrated our theoretical results with several numerical simulations.}
\end{abstract}


\textbf{AMS:}{ 35J70, 65J20, 35K65, 68U10.}\\

\textbf{Acknowledgment:} I warmly thank Antonin Chambolle who suggested me this fruitful research project.

\thispagestyle{plain}
\markboth{Khalid Jalalzai}{Discontinuities of solutions of the weighted
anisotropic TV}

\section{Introduction}
Functions of bounded variation equipped with the total variation semi-norm were introduced for image reconstruction in 1992. Since then, they have had many successful applications for inverse problems in imaging. Indeed, the penalization of the total variation has the ability to smooth out the image by creating large regular zones and to keep the edges of the most important objects in the image. In this paper we aim to study the first of these two key properties in the continuous setting and for general energies.

We assume that a corrupted image $g:\O\subset\R^2\to\R$ went through a degradation
\begin{align*}
g=g_0+n
\end{align*}
where $g_0$ is the original clean image, $n$ is a Gaussian white noise of standard deviation $\sigma$. Rudin, Osher and Fatemi (ROF) proposed in 1992 to minimize the total variation
\begin{align*}
u\mapsto TV(u)=\int_\O|Du|
\end{align*}
amongst functions of bounded variation under the constraint $\|u-g\|_2^2\leq\sigma^2|\O|^2$ to solve the inverse problem and thus get a restored image $u$.
It was proven in \cite{ChamLions} that one can solve in an equivalent way the unconstrained problem
\begin{align*}
\min_{u\in\BV}\la\int_\O|Du|+\frac{1}{2}\|u-g\|_2^2
\end{align*}
for an adequate Lagrange multiplier $\la$. In the literature the minimization of ROF's energy is referred to as the \textit{denoising problem}.


It has been long observed that using the total variation has the advantage of recovering the discontinuities quite well.
We will devote our study to the behavior of the minimizer of the denoising problem at these discontinuities.  
Recently Chambolle, Caselles and Novaga proved in \cite{ChamJump} that the discontinuities of the denoised image are contained in those of the datum $g$. In other words minimizing ROF's energy does not create new discontinuities. The idea of their proof is to use the coarea formula to look at the level sets of the minimizer locally and to detect the creation of jumps when two of these level sets touch. The argument was further refined by the same authors in \cite{ChamRegularity} to prove local Hölder continuity of the minimizer of the denoising problem when the datum is itself Hölder continuous. All these results are extended to the case of the total variation flow in both papers. See also \cite{JalalzaiFlow} where the initial datum $g$ is not assumed to be bounded.
After recalling classical notions on BV functions in \linebreak Section 2 and proving some basic facts on the level sets of the anisotropic TV in Section 3, our aim in Section 4 is to generalize the results of \cite{ChamJump} to a problem of the form
\begin{align*}
\min_{u\in\BV}\int_\O\Phi(x,Du)+\int_\O \Psi(x,u(x))dx.
\end{align*}
Here $\Phi$ is a smooth elliptic anisotropy, $\Psi$ is essentially strictly convex, coercive in the second variable and integrable in the first one. To adapt the argument of \cite{ChamJump} we need to recall some basic results on the regularity of solutions of elliptic PDEs and also some standard facts on the regularity of minimal surfaces \cite{Almgren77}.

Moreover, in Section 5 we refine the previous results in the weighted case
\begin{align*}
\min_{u\in\BV}\int_\O w(x)|Du|+\frac{1}{2}{\|u-g\|}_2^2
\end{align*}
and prove that whenever $w$ is merely Lipschitz, one can observe the creation of discontinuities, that is to say, the minimizer has discontinuities that are not contained in those of the datum $g$. We also prove that the jump (think of the contrast for images) is decreased at the discontinuity. This is quite counter-intuitive if one considers a datum that is highly oscillating in the neighborhood of the discontinuity. Our result is a key step in \cite{JalalzaiFlow}, allowing the authors to extend the results of \cite{ChamJump}. In the weighted case, we also prove directly the regularity of the level lines. The proof we provide is simple and does not rely on the theory of currents.

In Section 6, we discuss an open problem, precisely, whether the discontinuity sets of the solutions of the ROF's model form a decreasing sequence with respect to the regularization parameter.

Another very important property of the total variation is that it smoothes the highly oscillating regions by creating large constant zones which is known in the literature as the \textit{staircasing effect} and is sometimes not desirable. We investigate further this property in \cite{JalalzaiStairc,JalalzaiPhD} (see also the references therein). In particular, an interesting question is to understand how the staircase zones and the discontinuities evolve with the regularization parameter $\la$. The idea is to use the results that are already established for the flow. Unfortunately, in higher dimension the connection between the flow and ROF's energy fails. 
However, we prove in \cite{JalalzaiStairc,JalalzaiPhD} that this connection actually holds for radial functions. As a consequence the discontinuities form a decreasing sequence, whereas the staircase zones increase with the regularization parameter (which is not true in general). 

Let us remark that all the results of this paper are established in dimension $N\geq 2$ since the situation is well understood in the one-dimensional case and was widely studied in the literature (see the recent paper \cite{Bonforte} for instance). Indeed, in dimension one, ROF's denoising problem reads as follows
\begin{align}\label{chapTV:ROF_1_D}
\min_{u\in BV(\R)} \int \la|u'(x)|+\frac{1}{2}{(u-g)}^2(x)dx
\end{align}
for $g\in L^2(\R)$ and some positive real $\la$. Let us denote $u_\la$ the minimizer of this problem.

Writing down the Euler-Lagrange (see \cite{ChamTV}) one immediately sees
that either $u_\la$ is constant or $z_\la=\sgn(u_\la')$ and as a consequence $u_\la=g$. This is an almost explicit formulation of the solution that tells us that
\begin{itemize}
\item[-] the discontinuities of $u_\la$ are contained in those of $g$,
\item[-] flat zones are created at maxima and minima of $g$.
\end{itemize}
This can be seen in the following simulation:
\begin{figure}[H]   \begin{minipage}[c]{\linewidth}
     \includegraphics[width=\hsize,height=0.1\vsize]{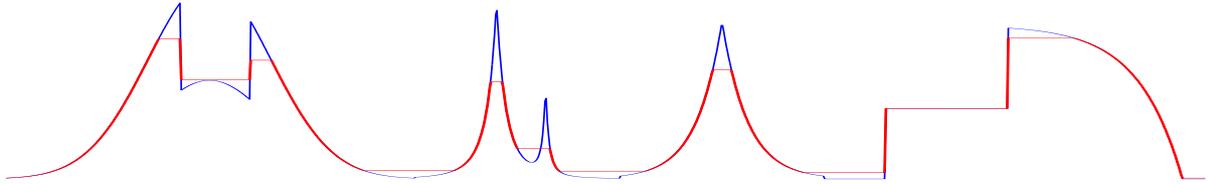}
    \caption
{Minimizer $u_\la$ (in red) of a 1D data $g$ (in blue).}
  \end{minipage}
\end{figure}

\section{Mathematical preliminary}
Henceforth $\O$ will denote an open subset of $\R^N$ with Lipschitz continuous boundary. The material of this section can be found in the classical textbooks~\cite{Ambrosio00,Giusti,Ziemer} but also in the recent survey~\cite{ChamTV}. 
\subsection{Functions of bounded variation}
Let us start with the following fundamental definition:
\begin{deff}
A function $u\in L^1(\O)$ is of \textit{bounded variation} in $\O$ (denoted $u\in BV(\O)$) if its distributional derivative $Du$ is a vector-valued Radon measure that has finite total variation \textit{i.e.} $|Du|(\O)<\infty$.
By the Riesz representation theorem, this is equivalent to say that
$$|Du|(\Omega)=\sup\left\{\int_\O u\diverg\vphi\ /\ \vphi\in C_c^\infty(\O,\R^N), \forall x\in\O\ |\vphi(x)|\leq 1\right\}<\infty.$$
In the sequel, the quantity $|Du|(\O)$ also denoted $\int_\O|Du|$ or simply $TV(u)$ will be called the \textit{total variation} of $u$. It is readily checked that ${\|\cdot\|}_1+TV$ defines a norm on $\BV$ that makes it a Banach space.
\end{deff}
A straightforward consequence of the dual definition we just gave is a key step to apply the direct method:
\begin{prop}[Sequential lower semicontinuity]\label{chapTV:lsc}
Let ${(u_n)}_{n\in\N}$ be any sequence in $BV(\O)$ such that $u_n\to u$ in $L^1(\O)$ then
\begin{align*}
\int_\O|Du|\leq\liminf_{n\to\infty} \int_\O|Du_n|.
\end{align*}
\end{prop}
One also has
\begin{prop}[Approximation by smooth functions]\label{chapTV:approxBV}
If $u\in\BV$ then there exists a sequence ${(u_n)}_{n\in\N}$ of functions in $C^\infty(\O)$ such that
\begin{align*}
u_n\to u\text{ in } L^1(\O)\\
\int_\O|\nabla u_n| \to \int_\O |Du|.
\end{align*}
\end{prop}
For the direct method to apply it is useful to have a compactness result:
\begin{thm}[Rellich's compactness in $BV$]\label{chapTV:Rellich}
Given a bounded $\O\subset\R^N$ with Lipschitz boundary and any sequence ${(u_n)}_{n\in\N}$ such that $\left({\|u_n\|}_{L^1(\O)}+\int_\O |Du_n|\right)$ is bounded,
there exists a subsequence ${(u_{n(k)})}_{k\in\N}$ that converges in $L^1$ to some $u\in\BV$ as $k\to\infty$.
\end{thm}
\begin{deff}
Let $E\subset\R^N$ be a Borelian set. It is called a \textit{set of finite perimeter} or also \textit{Caccioppoli set} if $u=\chi_E$ is a function of bounded variation. We will call \textit{perimeter} of $E$ in $\O$, and denote $P(E,\O)$ or simply $P(E)$, its total variation.
\end{deff}
The following key result provides a connection between the total variation of a function and the perimeter of its level sets.
\begin{thm}[Coarea formula]\label{chapTV:CoareaAFP}
If $u\in\BV$, the set $E_t=\{u>t\}$ has finite perimeter for a.e. $t\in\R$ and
$$|Du|(B)=\int_{-\infty}^\infty|D\chi_{\{u>t\}}|(B)dt$$
for any Borel set $B\subset\O$.
\end{thm}
%
Functions of bounded variation have some nice structural properties that we are going to recall here:
\begin{deff}\label{chapTV:def_jump}
We say that $u\in L^1_{loc}(\O)$ has an \textit{approximate limit} at $x\in\O$ if there exists $z\in\R$ such that
\begin{align*}
\lim_{r\to 0}\frac{1}{|B(x,r)|}\int_{B(x,r)}|u(y)-z|dy=0.
\end{align*}
The set of points where this does not hold is called the \textit{approximate discontinuity set} and denoted $S_u$.\\

We say that $x\in\O$ is an \textit{approximate jump point} of $u$ if there exist $u^+(x)\not = u^-(x)\in\R,\ \nu(x)\in\R^N$ a unitary vector such that
\begin{align*}
\lim_{r\rightarrow 0} \frac{1}{|B_r^\pm(x,\nu(x))|}\int_{B_r^\pm(x,\nu(x))} |u(y)-u^\pm(x)|=0
\end{align*}
where $B_r^\pm(x,\nu(x))=\{y\in B(x,r);\pm|\langle \nu(x),y-x \rangle|>0\}$.
We shall denote by $J_u$ the set of \textit{jump points}.\\

If $u=\chi_E$ is the characteristic function of a set $E$ of finite perimeter in $\O$, $J_u$ is then denoted $\partial^*E$ and called the \textit{reduced boundary} of $E$. 
\end{deff}
Then we have the following structure theorem:\\
\begin{thm}
If $u\in\BV$ then $\H^{N-1}(S_u\setminus J_u)=0$
and one has the following decomposition
\begin{align*}
Du=\G udx+(u^+-u^-)\nu\H^{N-1}_{|J_u}+D^c u
\end{align*}
for some measure $D^c u$ referred to as the Cantor part of $Du$.
\end{thm}

%

\subsection{BV functions in image processing}

The classical model of a functional where total variation plays a key role is the so-called Rudin-Osher-Fatemi energy:
\begin{align}\label{chapTV:ROF}
\tag{ROF}
\E_\lambda(u)=\lambda\int_\O|Du|+\frac{1}{2}{\|u-g\|}^2_2
\end{align}
In the sequel we shall be interested in minimizing this energy in $BV(\O)$ for some positive real $\lambda$. By Proposition \ref{chapTV:lsc}, there is a unique minimizer in $BV(\O)$, denoted $u_\lambda$ in the sequel.\\

The parameter $\la$ really plays the role of a tuning parameter as one can see it in the following
\begin{prop}\label{chapTV:ConvLambda}
Let $g\in L^2(\O)$, $\lambda$ some positive real and $u_\lambda$ be the corresponding minimizer of $(ROF)$ then whenever $\lambda\to 0$
$$u_\lambda\to g\text{ in }L^2(\O).$$
\end{prop}
\noindent In other words, the less we regularize the closer the minimizer gets to the data in the $L^2$ sense. This is quite what we expect. \\

The proof is really simple in case $g\in\BV$: 
since $g$ is itself a candidate for the minimization
\begin{align}\label{chapTV:ineg_g_BV}
\lambda\int_\O |Du_\lambda|+\frac{1}{2}{\|u_\lambda-g\|}_2^2&\leq\lambda\int_\O |Dg|
\end{align}
which yields the result with $\lambda\to0$.
In case $g\in L^2(\O)$, this proposition is actually a basic property of the proximal mapping (see~\cite[Proposition 2.6]{Brezis73}).\\ 

%
%

By the coarea formula, the superlevel sets $\{u_\lambda>t\}$ are sets of finite perimeter for almost every $t$ that satisfy the following minimal surface problem:\\
\begin{thm}\label{chapTV:LevelSetPb}
Let $u_\la$ be the minimizer of (\ref{chapTV:ROF}). Then for any $t\in\R$, $\{u_\la>t\}$ (resp. $\{u_\la\geq t\}$) is the minimal (resp. maximal) solution of the minimal surface problem
\begin{align}\label{chapTV:LevelSetPbEq}
\min_E \la P(E,\O)+\int_E \left(t-g(x)\right)dx
\end{align}
over all sets of finite perimeter in $\O$. Moreover $\{u_\la>t\}$ being defined up to negligible sets, there exists an open representative.
\end{thm}
The results of this section form the foundations for the study of similar properties for more general energies. This will be the object of the next part.


\section{Anisotropic total variation with a generic data fidelity}

\subsection{The anisotropic total variation: the case of a Finsler metric}
In calculus of variations, one is often interested in minimizing an integral functional of the form
$$\E(u)=\int_\O F(x,u,\nabla u)$$
among all $C^1(\O)$ or $W^{1,1}(\O)$ functions $u$ and under some additional constraints.
Though, the minimization problem is not well-posed and it is natural to seek for an extension of this functional in the completion of $W^{1,1}(\O)$ which is the space $BV(\O)$. Moreover, to apply the classical direct method of the calculus of variations we need the extension functional to be lower semicontinuous with respect to the $L^1$ convergence. The natural choice for the extension is therefore the so-called \textit{relaxed functional} $\bar{\E}(u)$ which corresponds to the lower semicontinuous envelope.\\

In the sequel, we shall be interested in minimizing energies of the form
\begin{align*}
\E(u)=\int_\O \Phi(x,\nabla u(x))dx+\int_\O \Psi(x,u(x))dx
\end{align*}
for some $\Phi$ and $\Psi$ that we will specify later. 
For the moment, we are going to focus on the first term namely 
\begin{align*}
J_\Phi(u)=\int_\O \Phi(x,\nabla u(x))dx
\end{align*}
and will recall its lower semicontinuous envelope $\bar{J}_\Phi$.\\

From now on, the integrand $\Phi(x,p):\O\times \R^N\rightarrow\R$ will be called \textit{a Finsler integrand} if


\begin{itemize}
\item[$\textbf{(H1)}$] $\Phi(x,\cdot)$ is convex for any $x\in\O$, 
\item[$\textbf{(H2)}$] $\Phi(x,\cdot)$ is of linear growth uniformly $x\in\O$ \emph{i.e.} 
$$C_\Phi^{-1}|p|\leq \Phi(x,p)\leq C_\Phi|p|,\ \forall x\in\O,\ p\in\R^N$$
for some positive real constant $C_\Phi$, 
\item[$\textbf{(H3)}$] $\Phi$ is positively 1-homogeneous in the variable $p$ \emph{i.e.} 
$$\Phi(\cdot,\lambda p)=\lambda \Phi(\cdot,p),\ \forall \lambda>0,\ p\in\R^N,$$
\item[$\textbf{(H4)}$] $\Phi$ is continuous.
\end{itemize}
The integrand $\Phi$ is a \textit{reversible Finsler integrand} if it satisfies in addition
\begin{itemize}
\item[$\textbf{(H5)}$] $\Phi(x,-p)=\Phi(x,p),\ \forall x\in\O,\ p\in\R^N.$
\end{itemize}
If one assumes that
\begin{itemize}
\item[$\textbf{(H6)}$] $\Phi(\cdot,p)$ and $D_p\Phi(\cdot,p)$ are Lipschitz continuous on $\O$ uniformly $p\in\S^{N-1}$ \textit{i.e.}
\begin{align*}
\sup_{p\in\S^{N-1}}&|\Phi(x,p)-\Phi(\tilde{x},p)|\leq C|x-\tilde{x}|\ \forall x,\tilde{x}\in \O,\\
\sup_{p\in\S^{N-1}}&|D_p\Phi(x,p)-D_p\Phi(\tilde{x},p)|\leq C|x-\tilde{x}|\ \forall x,\tilde{x}\in \O,
\end{align*}

\item[$\textbf{(H7)}$] $\Phi(x,\cdot)$ has locally $\beta$-Hölder second order partial derivatives (with $\beta\in(0,1]$) on $\R^N\setminus\{0\}$ uniformly $x\in\O$ 
and
\begin{align*}
|D^2_p\Phi(x,p)|\leq C\ \forall x\in\O,\ p\in \S^{N-1},
\end{align*}

\item[$\textbf{(H8)}$] $\Phi$ is \textit{elliptic} in the sense that
\begin{align*}
\langle D_p^2\Phi(x,p)\xi,\xi\rangle\geq\frac{{\left|\xi-\left(\xi\cdot\frac{p}{|p|}\right)\frac{p}{|p|}\right|}^2}{|p|},\ \forall x\in\O,\ \xi\in\R^N,\ p\in\R^N\setminus\{0\},
\end{align*}
or equivalently (see \cite{Almgren77})
\begin{align*}
\langle \nabla_p \Phi(x,p)- \nabla_p \Phi(x,\tilde{p}),p-\tilde{p}\rangle\geq {|p-\tilde{p}|}^2,\ \forall x\in\O,\ p,~\tilde{p}\in\S^{N-1},
\end{align*} 
\end{itemize}
we shall say that $\Phi$ is \textit{a strongly convex Finsler integrand}.\\

\begin{rmqs}\label{chapTV:Rmq_def_Phi}
\begin{trivlist}
\item[$(i)$]Assumptions $(H1)-(H3)$ imply that $\Phi(x,\cdot)$ is Lipschitz uniformly $x\in\O$ \textit{i.e.}
$$\sup_{x\in\O} |\Phi(x,p)-\Phi(x,\tilde{p})|\leq C |p-\tilde{p}|.$$
\item[$(ii)$]For later reference we also note that $p\cdot \G_p\Phi(\cdot,p)=\Phi(\cdot,p)$ by assumption $(H3)$. 
If $\beta=1$, then by $(H7)$ it follows that $\G_p\Phi$ is Lipschitz continuous on $\R^N\times\S^{N-1}$. 
\item[$(iii)$]Any Riemannian metric or more generally Finsler metric gives rise to a Finsler integrand. We refer to \cite{Amar,Bellettini} for further details.
\item[$(iv)$] We get the total variation by simply setting $\Phi(x,p)=|p|$ which is also called in the literature the \textit{area integrand.} \\
\end{trivlist}
\end{rmqs}

As we just said, to be of some interest the extension of functional $J_\Phi$ has to be lower semicontinuous on $BV(\O)$. This is ensured by the following representation result:

\begin{prop} Let $\Phi:\O\times \R^N\rightarrow\R$ be a Finsler integrand. For any $u\in\BV$ we have
$$\bar{J_\Phi}(u)=\int_\O \Phi(x,\nabla u)+\int_\O \Phi\left(x,\frac{D^s u}{|D^s u|}\right)|D^s u|$$
where $\frac{D^s u}{|D^s u|}$ is the Radon-Nikodym derivative of $D^su$ with respect to $|D^s u|$.\\
\end{prop}

One of the first versions of this theorem was proven by Demengel and Temam in~\cite{Demengel} in case $J_\Phi(\mu)=\int \Phi(\mu)$. The result was progressively refined in articles such as~\cite{Bouchitte} and \cite{Aviles} which contains the general case in which we are interested. We would like to note that the hypotheses we made can be weakened a little (see \cite{Ambrosio00}). In \cite[Theorem 5.1]{Amar}, it is proven that the latter definition has a dual counterpart:

\begin{prop}\label{chapTV:DefDualeBV_f}
Let $\Phi$ be a Finsler integrand and $u\in\BV$ then
\begin{align*}
\bar{J_\Phi}(u)=\sup\left\{\int_{\O} u \diverg\vphi\ /\ \vphi\in C^1_c(\O,\R^N),\ \forall x\in\O\ \Phi^0(x,\vphi(x))\leq 1\right\}
\end{align*}
where $\Phi^0$ denotes the polar of $\Phi$ defined by
\begin{align*}
\Phi^0(x,\vphi(x))=\max\{p\cdot\vphi(x)\ /\ p\in\R^N, \Phi(x,p)\leq 1\}.\\
\end{align*}
\end{prop}
Henceforth we will not make any distinction between $J_\Phi$ and its $L^1$-lower semicontinuous envelope $\bar{J_\Phi}$.\\

\begin{deff}
Let $\Phi$ be a Finsler integrand. If $u\in\BV$ then the quantity $J_\Phi(u,\O)$ (or simply $J_\Phi(u)$ if $\O=\R^N$) is the \textit{anisotropic total variation} of $u$ in $\O$. If $E$ is a set of finite perimeter then the \textit{anisotropic perimeter} of $E$ in $\O$, denoted $P_\Phi(E,\O)$ (or $P_\Phi(E)$ if $\O=\R^N$) is the anisotropic total variation of $E$ namely
\begin{align*}
P_\Phi(E,\O)=\int_{\partial^*E\cap\O} \Phi(x,\nu_E)d\mathcal{H}^{N-1}(x).
\end{align*}
\end{deff}

\begin{rmq}\label{chapTV:per_compl}
If $\Phi$ is a reversible Finsler integrand then for any set $E$ of finite perimeter in $\O$ one has
$P_\Phi(E,\O)=P_\Phi(\O\setminus E,\O).$
\end{rmq}
Soon, we will need two generalizations of the coarea formula for the anisotropic total variation. Let us state them here:
\begin{prop}\label{chapTV:CoareaWeight}
Let $u\in\BV$ and $w:\O\rightarrow\R$ be a non-negative Borelian weight. Then one has
$$\int_\O w|Du|=\int_{-\infty}^\infty\left(\int_\O w|D\chi_{\{u>t\}}|\right)dt=\int_{-\infty}^\infty P_w(\{u>t\},\O)dt.$$
\end{prop}

\begin{proof}
By \cite[Theorem 7]{Evans92}, there exists a sequence of Borelian sets ${(A_k)}_{k\in\N}$ such that
$$w=\sum_{k=1}^\infty \frac{1}{k}\chi_{A_k}.$$
Therefore by Fubini and then by Theorem \ref{chapTV:CoareaAFP},
\begin{align*}
\int_\O w|Du|&=\sum_{k=1}^\infty\frac{1}{k}\int_{A_k}|Du|=\sum_{k=1}^\infty\frac{1}{k}\int_{-\infty}^\infty|D\chi_{\{u>t\}}|(A_k)dt\\
&=\int_{-\infty}^\infty\left(\sum_{k=1}^\infty\frac{1}{k}\int_{A_k}|D\chi_{\{u>t\}}|\right)dt=\int_{-\infty}^\infty\left(\int_\O w|D\chi_{\{u>t\}}|\right)dt.
\end{align*}
\end{proof}

The following proposition is stated in \cite[Remark 4.4]{Amar} without any proof:
\begin{prop}\label{chapTV:CoareaAnisotropic}
Let $u\in\BV$ and $\Phi:\O\times\R^N\rightarrow\R$ be a Finsler integrand. Then then one has
$$\int_\O \Phi(x,Du)=\int_{-\infty}^\infty P_\Phi(\{u>t\},\O)dt.$$
\end{prop}

\begin{proof}
Applying Proposition \ref{chapTV:approxBV}, one can pick an approximating sequence ${(u_n)}_{n\in\N}$ of $C^\infty$ functions such that $u_n\to u$ in $L^1(\O)$ and $\int_\O|\nabla u_n| \to \int_\O |Du|$.\\
If we set $$w(x)=\Phi\left(x,\frac{\nabla u_n(x)}{|\nabla u_n(x)|}\right)$$ whenever $\nabla u_n(x)\not=0$ then by Proposition \ref{chapTV:CoareaWeight} the result holds for $u_n$, namely
\begin{align*}
\int_\O \Phi\left(x,\frac{\nabla u_n(x)}{|\nabla u_n(x)|}\right)|\nabla u_n(x)|dx&=\int_{-\infty}^\infty\left(\int_\O \Phi\left(\cdot\ ,\frac{\nabla u_n}{|\nabla u_n|}\right)|D\chi_{\{u_n>t\}}|\right)dt\\
&=\int_{-\infty}^\infty P_\Phi(\{u_n>t\},\O)dt.
\end{align*}
Finally by Reshetnyak Theorem 2.39 in \cite{Ambrosio00}, we can send $n\to+\infty$ and we get
\begin{align*}
\int_\O \Phi\left(\cdot,\frac{Du}{|Du|}\right)d|Du|&\geq \liminf_{n\to\infty} \int_{-\infty}^\infty P_\Phi(\{u_n>t\},\O)dt\\
&\geq  \int_{-\infty}^\infty P_\Phi(\{u>t\},\O)dt
\end{align*}
where in the second line we used the lower semicontinuity of $P_\Phi$ in conjunction with Fatou's lemma.

To prove the converse inequality, 
let us pick a candidate $\vphi\in C^1_c(\O,\R^N)$ such that for any $x\in\O$, $\Phi^0(x,\vphi(x))\leq 1$. Then, by the layer cake formula and by application of Fubini and Proposition \ref{chapTV:DefDualeBV_f},
\begin{align*}
\int_\O u\diverg\vphi&=\int_{-\infty}^{\infty}\int_\O \chi_{\{u>t\}}(x)\diverg\vphi(x)\ dxdt\\
&\leq\int_{-\infty}^{\infty}P_\Phi\left(\{u>t\},\O\right)dt,
\end{align*}
which proves the result taking the supremum of the left hand side over all admissible $\vphi$.
\end{proof}

\subsection{The minimization problem for functions}
In the sequel, we are going to consider the following energy 
\begin{align}\label{chapTV:ROF_f}
\E(u)=\int_\O \Phi(x,Du)+\int_\O \Psi(x,u(x))dx 
\end{align}
over the space $\BV$. Henceforth, we assume that $\Phi$ is a Finsler integrand and that 
\newline
\begin{itemize}
\item[$\textbf{(H9)}$] $\Psi(x,t):\O\times\R\rightarrow\R$ is measurable in $x$, strictly convex and coercive in $t$, that is to say 
\begin{align*}
\lim_{t\to\pm\infty}\Psi(x,t)=+\infty,
\end{align*}
and such that
\begin{align}\label{chapTV:hyp_psi0}
\Psi(\cdot,0)\in L^1(\O),
\end{align}
\begin{align}\label{chapTV:hyp_psi}
\partial_t^{-}\Psi(\cdot,t)\in L^1(\O)\ \forall t\in\R.
\end{align}

\end{itemize}

\begin{rmq}\label{chapTV:rmq_cvx}
\begin{trivlist}
\item [$(i)$] We recall that function $\Psi(x,\cdot)$ being convex for any $x\in \R^N$ it is therefore locally Lipschitz continuous on $\R$ (see \cite{Ekeland} for instance). We therefore denote 
\begin{align*}
\partial_t\Psi(x,t):=\partial_t^-\Psi(x,t)
\end{align*}
the left derivative that exists at any $t\in\R$. 
\item [$(ii)$] Clearly the energy (\ref{chapTV:ROF}) is a special case of (\ref{chapTV:ROF_f}) since it amounts to take $\Phi(x,p)=|p|$ and $\Psi(x,u(x))=\frac{1}{2}{(u(x)-g(x))}^2$ for some $g\in L^2(\O)$ and bounded $\O$. Observe that one can also consider a general data fidelity term of the form $\Psi(x,u(x))=\frac{1}{q}{(u(x)-g(x))}^q$ for some $g\in L^q(\O)$ with $q>1$.
\item[$(iii)$] Let us note that for $t>s$ and $x\in\O$,
\begin{align}
\partial_t\Psi(x,s)(s-t)\leq\Psi(x,s)-\Psi(x,t)\leq \partial_t\Psi(x,t)(t-s)
\end{align}
which, in conjunction with (\ref{chapTV:hyp_psi0}), implies that
\begin{align*}
\Psi(\cdot,t)\in L^1(\O)\ \forall t\in\R.
\end{align*}
For further reference, let us also remark that if $t_n\to t\in\R$ then for $n$ large,
\begin{align*}
|\Psi(x,t_n)-\Psi(x,t)|\leq \sup_{k\geq n}|t_k-t|\max(|\partial_t\Psi(x,t-1)|,|\partial_t\Psi(x,t+1)|)
\end{align*}
hence in particular $\Psi(\cdot,t_n)\to\Psi(\cdot,t)$ in $L^1(\O)$.
\item [$(iv)$] We could have replaced assumption (\ref{chapTV:hyp_psi}) in $(H9)$ by
\begin{align*}
\Psi(\cdot,t)\geq \psi\in L^1(\O)\ \forall t\in\R,
\end{align*}
if one considered \textit{local minimizers} of the ROF problem on an unbounded domain but this would lead us too far. See the beginning of \cite[Section 5]{ChamRegularity} for further details.\\
\end{trivlist}
\end{rmq}

Applying the direct method, we get readily
\begin{prop} Let $\Phi$ be a Finsler integrand and $\Psi$ measurable in $x$ and strictly convex in the second variable, then
$\E$ has a unique minimizer $u$ in the space $\BV$.
\end{prop}
\begin{proof}
Consider a minimizing sequence ${(u_n)}_{n\in\N}$ such that $\E(u_n)\to \inf_\O \E$. As $\E(u_n)\leq \E(0)<+\infty$, assumption $(H2)$ implies that ${(u_n)}_{n\in\N}$ is bounded in $BV(\O)$. Then, Rellich's theorem asserts that up to extraction of a subsequence (still denoted ${(u_n)}_{n\in\N}$) it converges in $L^1(\O)$ and also pointwise to some $u\in \BV$. By lower-semicontinuity of $J_\Phi$ and Fatou, we get
\begin{align*}
\E(u)\leq \liminf_{n\to\infty }\E(u_n)=\inf_\BV \E(u).
\end{align*}
This proves the existence of a minimizer namely $u$. It is unique by strict convexity of $\E$.
\end{proof}

\begin{rmq}
It is also possible to reason in a slightly different way to avoid using Rellich's theorem in case $\Psi(x,u(x))=\frac{1}{q}(u(x)-g(x))^q$. This way we also avoid the regularity assumption on $\partial\O$ (see \cite{ChamTV,ChamPNL}). 
\end{rmq}

\subsection{The minimization problem for level sets}
We assume, for the time being, that $\Phi$ is a Finsler integrand and $\Psi$ is as above. 
Let us introduce the following minimal surface problems parametrized by $t\in\R$
\begin{align}\label{chapTV:LevelSetPb_f}
\min_E P_\Phi(E,\O)+\int_E \partial_t\Psi(x,t)dx.
\end{align}
The minimization is carried out on all sets of finite perimeter.
Simply reasoning as in the previous proof we get the existence of minimizers (and again, in some cases, it is possible to avoid using Rellich's theorem as was done in \cite{ChamPNL}). Obviously, we may not have a unique solution. Given $t\in\R$, we shall denote $E_t$ a solution of (\ref{chapTV:LevelSetPb_f}).\\

The following comparison result similar to \cite[Lemma 2.1]{ChamMCM} and \cite[Lemma 4]{Alter} will be needed:\\
\begin{lem}\label{chapTV:lem_comparaison}
Let $f_1,f_2\in L^1(\O)$ and $E$, $F$ be respectively minimizers of
\begin{align*}
\min_{E} P_\Phi(E,\O)-\int_{E} f_1(x)dx\ \ \ \text{ and }\ \ \ \min_{F} P_\Phi(F,\O)-\int_F f_2(x)dx
\end{align*}
Then, if $f_1<f_2$ a.e., $|E\setminus F|=0$ (\textit{i.e.} $E\subset F$ up to a negligible set).
\end{lem}
\begin{proof}
First, observe that by the triangle inequality
$$J_\Phi(\chi_E+\chi_F,\O)\leq P_\Phi(E,\O)+P_\Phi(F,\O).$$
Whereas, by the coarea formula we also know that
\begin{align*}
J_\Phi(\chi_E+\chi_F,\O)&=\int_0^2P_\Phi(\{\chi_E+\chi_F>t\},\O)dt\\
&=P_\Phi(E\cup F,\O)+P_\Phi(E\cap F,\O).
\end{align*}
This proves that
\begin{align}\label{chapTV:InegPer}
P_\Phi(E\cap F,\O)+P_\Phi(E\cup F,\O)\leq P_\Phi(E,\O)+P_\Phi(F,\O).
\end{align}
Now, by minimality of $E$ and $F$, we get
\begin{align*}
P_\Phi(E,\O)-\int_E f_1(x)dx \leq P_\Phi(E\cap F,\O)-\int_{E\cap F} f_1(x)dx,\\
P_\Phi(F,\O)-\int_F f_2(x)dx \leq P_\Phi(E\cup F,\O)-\int_{E\cup F} f_2(x)dx.
\end{align*}
Adding both inequalities and using (\ref{chapTV:InegPer}), we have
$$\int_{E\setminus F}\left(f_1(x)-f_2(x)\right)dx\geq 0$$
hence the result since $f_1<f_2$ a.e.
\end{proof}
In particular, we observe that\\
\begin{lem}\label{chapTV:NivCroissants}
If $t<t'$ and $E_{t}$, $E_{t'}$ are the corresponding minimizers of the minimal surface problem (\ref{chapTV:LevelSetPb_f}) then 
$E_{t'}\subset E_{t}$ up to a negligible set.
\end{lem}
\begin{proof}
Note that the strict convexity implies $\partial_t\Psi(\cdot,t)<\partial_t\Psi(\cdot,t')$ thus the statement follows from the previous lemma.
\end{proof}

Knowing this we can, as in \cite{ChamTV}, introduce 
\begin{align*}
E^-_t=\bigcup_{t'>t} E_{t'},\ E^+_t=\bigcap_{t'<t} E_{t'},
\end{align*}
respectively the smallest solution and largest solution of
\begin{align*}
\min_E P_\Phi(E,\O) + \int_E \partial_t \Psi(x,t)dx.
\end{align*}
One has to be careful because the sets $E_{t'}$ are defined up to negligible sets. As a consequence, the non-countable union and intersection may not be well-defined. To remedy this problem one could have taken as a representative for $E_{t'}$ the set of points of density 1 which is also an open set. This can be shown thanks to the density lemma for the anisotropic perimeter (see \cite{Caffarelli,ChamThou}). We will come back to this later (see Lemma \ref{chapTV:density} and the remark that follows).\\

If we set 
\begin{align*}
v(x):=\sup\{t\in\R\ /\ x\in E_t\}.
\end{align*}
It is easily seen that $\{v>t\}=E^-_t$ and that $\{v\geq t\}=E^+_t$. 
\newline\newline
Proceeding as for the total variation (see \cite{ChamMCM}), we get\\
\begin{lem}
Let 
$\Phi$ be a reversible Finsler integrand and $\Psi$ as in $(H9)$. Then
$v$ is the minimizer of $\E$.
\end{lem}
\begin{proof}
\textit{Step 1.} We claim that $\Psi(\cdot,v)\in L^1(\O)$. Let us prove it. Since $E^-_t$ solves (\ref{chapTV:LevelSetPb_f}), we have in particular
\begin{align}\label{chapTV:per_neg}
P_\Phi(E_t^-,\O)+\int_{E_t^-} \partial_t\Psi(x,t)dx\leq 0.
\end{align}
Integrating with respect to $t$ it follows
\begin{align*}
\int_0^M \int_{E_t^-} \partial_t{\Psi}(x,t)dxdt\leq 0
\end{align*}
where by Fubini's theorem the integral to the left can be rewritten
\begin{align*}
\int_0^M \int_{E_t^-} \partial_t{\Psi}(x,t)dxdt&=\int_{E_0^-}\int_0^{\min(v(x),M)}\partial_t{\Psi}(x,t)dtdx\\
&=\int_{E_0^-}\big({\Psi}(x,\min(v(x),M))-{\Psi}(x,0)\big)dx
\end{align*}
which implies
\begin{align*}
\int_{\{v>0\}}{\Psi}(x,\min(v(x),M))dx\leq\int_{\O}{\Psi}(x,0)<+\infty.
\end{align*}
Now, by Fatou's lemma, that can by applied by Remark \ref{chapTV:rmq_cvx},
\begin{align*}
\int_{\{v>0\}}{\Psi}(x,v(x))dx<+\infty.
\end{align*}
Observe that by Remark \ref{chapTV:per_compl} (this is where the reversibility of $\Phi$ comes into play), $\{-v>t\}=\{v<-t\}=\O\setminus E_{-t}^+$ solves
\begin{align*}
\min_E P_\Phi(E,\O)-\int_{E} \partial_t\Psi(x,-t)dx.
\end{align*}
Thus replacing $\Psi(\cdot,t)$ by $\Psi(\cdot,-t)$ in (\ref{chapTV:LevelSetPb_f}), function $v$ is replaced by $-v$.
This proves
\begin{align*}
\int_{\{v<0\}}{\Psi}(x,v(x))dx<+\infty
\end{align*}
hence our claim.\\\\
\textit{Step 2.}Minimization property.\\
Let $v'\in\BV$ such that $\Psi(\cdot,v')\in L^1(\O)$ a candidate for the minimization and denote $E'_t=\{v'>t\}$ for some $t\in\R$. By minimality of $E^-_t$,
\begin{align*}
P_\Phi(E_t^-,\O) + \int_{E_t^-} \partial_t \Psi(x,t)dx\leq P_\Phi(E_t',\O) + \int_{E_t'} \partial_t \Psi(x,t)dx. 
\end{align*}
Integrating with respect to $t$
\begin{align}\label{chapTV:int_per}
\int_{-M}^M \left(P_\Phi(E_t^-,\O) + \int_{E_t^-} \partial_t \Psi(x,t)dx\right) dt\leq \int_{-M}^M \left(P_\Phi(E_t',\O) + \int_{E_t'} \partial_t \Psi(x,t)dx\right)dt. 
\end{align}
Note that by Fubini's theorem
\begin{align*}
\int_{-M}^M\int_{E_t^-} \partial_t \Psi(x,t)dx dt
&=\int_{\O}\int_{-M}^M\chi_{\{v>t\}}\partial_t \Psi(x,t)dtdx\\
&=\int_{\O}\int_{\min(v(x),-M)}^{\min(v(x),M)}\partial_t{\Psi}(x,t)dtdx\\
&=\int_{\O}\Psi(x,\min(v(x),M))-\Psi(x,\min(v(x),-M))dx\\
&=\int_{\O}\Psi(x,v(x))dx-\int_{\O}\Psi(x,-M)dx+\mathcal{R}(v,M)
\end{align*}
where we set 
\begin{align*}
\mathcal{R}(v,M)=\int_{\O}\Psi(x,\min(v(x),M))-\Psi(x,v(x))dx\\
+\int_{\O}\Psi(x,-M)-\Psi(x,\min(v(x),-M))dx.
\end{align*}
For the function $v'$ one obtains a similar identity, namely
\begin{align*}
\int_{-M}^M\int_{E_t'} \partial_t \Psi(x,t)dx dt
=\int_{\O}\Psi(x,v'(x))dx-\int_{\O}\Psi(x,-M)dx+\mathcal{R}(v',M).
\end{align*}
Though, for any function $v$ such that $\Psi(\cdot,v)\in L^1$
\begin{align*}
\lim_{M\to+\infty}\mathcal{R}(v,M)=0.
\end{align*}
Indeed, observe that on $\{v>M\}$
\begin{align*}
|\Psi(\cdot,\min(v,M))-\Psi(\cdot,v)|=\Psi(\cdot,v)-\Psi(\cdot,M)\leq\Psi(\cdot,v),
\end{align*}
and on $\{v<-M\}$
\begin{align*}
|\Psi(\cdot,-M)-\Psi(\cdot,\min(v,-M))|=\Psi(\cdot,v)-\Psi(\cdot,-M)\leq\Psi(\cdot,v),
\end{align*}
which proves the claim by application of the dominated convergence theorem.\\

In the end, making $M\to+\infty$ in (\ref{chapTV:int_per}) and using the anisotropic coarea formula we get
\begin{align*}
\int_\O \Phi(x,Dv)+\int_\O \Psi(x,v(x))dx\leq\int_\O \Phi(x,Dv')+\int_\O \Psi(x,v'(x))dx
\end{align*}
which means that $v$ minimizes $\E$ that is $v=u$ since the minimizer is unique.
\end{proof}
As a consequence, one actually proved\\
\begin{prop}\label{chapTV:LevelSetPb_f_thm}
Let 
$\Phi$ be a reversible Finsler integrand, $\Psi$ satisfy $(H9)$ and 
$u$ be the minimizer of $\mathcal{E}$.
Then the superlevel $E_t:=\{u>t\},\ t\in\R$, solves the minimal surface problem 
\begin{align*}
\min_E P_\Phi(E,\O)+\int_E \partial_t\Psi(x,t)dx
\end{align*}
over all sets of finite perimeter in $\O$.
\end{prop}
\begin{rmq}
The case $\Psi(x,t)=F(t-g(x))$ with $F$ of class $C^1$ and merely convex is discussed in \cite{Thouroude}. The proof is based on an approximation argument. We expect the argument to work for a general integrand $\Psi(x,t)$ convex in $t$. Though, for our future analysis, such a refinement is not necessary.
\end{rmq}

\section{The discontinuity set}\label{chapTV:sec_jump}
In this section, we are interested in proving qualitative results on the behavior of the jump set of the minimizer of (\ref{chapTV:ROF_f}). For this purpose, we first need to deal with the regularity of the level sets of the minimizer.

\subsection{Regularity theory for elliptic PDEs}
Let us recall a classical result that is taken from Gilbarg and Trudinger's book \cite{Gilbarg} (see in particular Theorem 8.9 and Theorem 9.15).
First, let us consider an operator in non divergence form
$$L=\sum_{i,j}a_{i,j}\partial_{x_i x_j}+\sum_i b_i\partial_{x_i}+c,$$
that satisfies the ellipticity condition
$$\sum_{i,j}a_{i,j}\xi_i\xi_j\geq C{|\xi|}^2.$$
We will say that $L$ is \textit{strictly elliptic}.\\

For such an operator, one has a general existence and regularity result for the Dirichlet problem:\\
\begin{thm}[{\cite[Theorem 9.15]{Gilbarg}}]\label{chapTV:thm_reg_edp}
\ Let $\O$ be a $C^{1,1}$ open domain in $\R^N$, and let the operator $L$ be strictly elliptic in $\O$ with coefficients $a_{i,j}\in C^0(\bar{\O})$, $b_i,\ c\in L^\infty(\O)$, with $i$, $j=1,\cdots, n$ and $c\leq 0$. Then if $f\in L^p(\O)$ and $\vphi\in W^{2,p}(\O)$ with $1<p<\infty$, the Dirichlet problem $Lu=f$ in $\O$, $u-\vphi\in W^{1,p}_0$ has a unique solution $u\in W^{2,p}(\O)$.\\
\end{thm}
Here me must mention the names of Ennio De Giorgi, John Nash and Jürgen Kurt Moser whose pioneering works contributed to the theory of regularity of elliptic PDEs. We refer to \cite{Mingione} for further PDE related regularity results and historical facts.\\

\begin{rmq}\label{chapTV:Morrey}
Let us also recall that for Sobolev spaces we have the following embedding in Hölder spaces
\begin{align*}
W^{k,p}(\O)\subset C^{r,\alpha}(\O)
\end{align*}
when $k-r-\alpha=\frac{N}{p}$ and $\alpha\in(0,1)$.\\
\end{rmq}
This is a consequence of Morrey's inequality (see \cite{Evans10}).
Consequently, if $p>N$ in the previous theorem, the solution inherits more regularity, namely $u\in C^{1,\alpha}$ with $\alpha=1-\frac{p}{N}$.

\subsection{Regularity issues for minimal surfaces}

The classical regularity theory for minimal surfaces (see \cite{Almgren77,Schoen,Bombieri}) and the recent paper \cite[Theorem 6.1]{Duzaar}, which discusses the regularity of rectifiable currents that are almost minimizers of an elliptic integrand, yield
\newline
\begin{thm}\label{chapTV:Regularity_f} Let $\Phi$ be any strongly convex Finsler integrand and $\Psi$ be such that assumption $(H9)$ is satisfied. We also assume that for some real $t$, ${\partial_t\Psi(\cdot,t)}\in L^p(\O)$ with $p>N$. Then a set $E_t$ that solves (\ref{chapTV:LevelSetPb_f}) has a reduced boundary $\partial^* E_t$ of Hölder class $C^{1,\alpha}$ for any $\alpha<\frac{1}{2}\left(1-\frac{N}{p}\right)$. \\
Moreover, $\partial E\setminus\partial^* E$ is a closed set and
$$\H^s(\partial E_t\setminus\partial^* E_t)=0$$
for every $s>N-3$.
\end{thm}
The hypothesis ${\partial_t\Psi(\cdot,t)}\in L^p(\O)$ with $p>N$ is essential. Indeed, in \cite{Barozzi}, the authors even prove that any set of finite perimeter $E\subset\R^N$ solves the prescribed mean curvature problem
\begin{align*}
\min_E P(E,\R^N)+\int_E f
\end{align*}
for some appropriate $f\in L^1(\R^N)$.\\

Morgan proved in \cite{Morgan91} that the value $N-3$ is sharp by exhibiting an example of a singular $\Phi$-minimizing hypersurface in $\R^4$.\\

When the anisotropy takes the form $\Phi(x,p)=w(x)|p|$ for some Hölder continuous weight $w$, it is possible to refine these regularity results and get $N-8$ instead of $N-3$ without even using the theory of currents. We shall discuss this case with many more details and references in Section \ref{chapTV:weighted_case}.\\

If one assumes in addition that $\partial_t\Psi(\cdot,t)\in L^\infty$  we can gain a little more regularity:
\begin{thm}\label{chapTV:Regularity+_f}
Let $\Phi$ be a strongly convex Finsler integrand that is Lipschitz continuous in $p$ uniformly $x$ 
and consider a function $\Psi$ that satisfies $(H9)$. Suppose that for some $t\in\R$, $\partial_t\Psi(\cdot,t)\in L^\infty$ and pick $E_t$ 
 that solves (\ref{chapTV:LevelSetPb_f}). Then $\partial^* E_t$ is $W^{2,p}$ for all $1<p<\infty$ and thus $C^{1,\alpha}$ for any $\alpha<1$.\\ 
In addition, $\partial E\setminus\partial^* E$ is closed and for every $s>N-3$
$$\H^s(\partial E\setminus\partial^* E)=0.$$
\end{thm}


We actually improve Theorem \ref{chapTV:Regularity_f} since the degree of Hölder continuity of the boundary increases from $\alpha/2$ to $\alpha$. This result is stated in \cite[p.140]{Ambrosio97}  for the classical curvature problem (\ref{chapTV:LevelSetPbEq}). As we could not find any precise reference for this more general case, we provide a proof.\\

First, let us point out that $\partial^* E_t$ is $W^{2,2}_{loc}$ as a consequence of the following lemma:

\begin{lem}\label{chapTV:lem_A}
Let $v\in C^1(\O')$ with $\O'\subset\R^{M}$ open be a weak solution of
\begin{align}\label{chapTV:MC_w}
-\diverg \left(A(\cdot,v,\G v)\right)=h 
\end{align}
where $h\in L^\infty(\O')$ and $A:\O'\times\R\times\R^{M}\rightarrow\R^{M}$ is Lipschitz continuous and locally strictly monotone \emph{i.e.} for any compact set $K\subset\R^{M}$ there is a constant $c_K$ s.t.
\begin{align}\label{chapTV:A_ell}
\langle A(x,t,p)-A(x,t,\tilde{p}),p-\tilde{p}\rangle \geq c_K{|p-\tilde{p}|}^2,\ \forall p,\tilde{p}\in K
\end{align}
uniformly $x\in\O',t\in\R$. Then $v\in W^{2,2}_{loc}(\O')$.
\end{lem}
\begin{rmq}\label{chapTV:rmq_A}
Notice that the mean curvature equation 
\begin{align}\label{chapTV:courb_prescr}
\diverg\left(\frac{Dv}{\sqrt{1+|Dv|^2}}\right)=h
\end{align}
is a special case of (\ref{chapTV:MC_w}) corresponding to 
\begin{align*}
A(x,t,p)=\frac{p}{\sqrt{1+|p|^2}}.
\end{align*}
The latter does satisfy the ellipticity condition (\ref{chapTV:A_ell}).\\
\end{rmq}

The proof of the lemma is based on Nirenberg's method (see \cite{Brezis83} for instance for further details). We simply adapt the proof given in \cite[Proposition 7.56]{Ambrosio00}:
\newline
\begin{proof}
Since the property we are interested is local, we can assume that ${\O'}$ is a ball of measure less than 1, that  ${\|\nabla v\|}_\infty\leq M_v$ for some positive $M_v$. We will consider that $K=\overline{B(0,M_v)}$ and will denote respectively $L_v,\ L_A$ the Lipschitz constants of $v,\ A$ in $K$ and $M_A$ the maximum of $A$ over $K$. For any generic function $u$ we denote the difference 
\begin{align*}
\Delta_\e u(x)=\frac{u(x+\e e_i)-u(x)}{\e}
\end{align*}
in the direction $e_i$, $i$ ranging from $1$ to $M$.

Now, in the weak formulation of (\ref{chapTV:MC_w}), we take as test functions $\vphi({\cdot-\e e_i})$ and $\vphi$ with $\vphi\in C^\infty_c({\O'})$, $\e>0$ small enough and substract the two identities to get after a change of variable
\begin{align*}
\frac{1}{\e}\int_{\O'}\langle A(x+\e e_i,v(x+\e e_i),\nabla v(x+\e e_i))-A(x,v(x),\nabla v(x)),\nabla\vphi(x)\rangle dx\\=-\int_{\O'} h\Delta_{-\e}\vphi
\end{align*}
which can be rewritten as
\begin{align*}
\frac{1}{\e}\int_{\O'}\langle \left[A(x,v(x),\nabla v(x+\e e_i))-A(x,v(x),\nabla v(x))\right],\nabla\vphi(x)\rangle dx\\
=-\frac{1}{\e}\int_{\O'}\langle\left[A(x+\e e_i,v(x+\e e_i),\nabla v(x+\e e_i))-A(x,v(x),\nabla v(x+\e e_i))\right],\nabla\vphi(x)\rangle dx\\
-\int_{\O'} h\Delta_{-\e}\vphi.
\end{align*}
We take $\vphi=\eta^2\Delta_\e v$ with $\eta\in C^1_c({\O'})$ and $0\leq\eta\leq1$ a cut-off function.
Notice that $\nabla\vphi=2\eta\nabla\eta\Delta_\e v +\eta^2\Delta_\e(\nabla  v)$
so using (\ref{chapTV:A_ell}) we may estimate the first integral from below by
$$ c_K\int_{\O'}\eta^2{|\Delta_\e(\nabla v)|}^2-2 L_AM_v{\|\nabla\eta\|}_\infty\int_{\O'}\eta|\Delta_\e(\nabla v)|.$$
The second integral can be controlled by
$$L_A\sqrt{1+M_v^2}\left(2M_v{\|\nabla\eta\|}_{\infty}+\int_{\O'}\eta^2|\Delta_\e(\nabla v)|\right).$$
As for the last integral, we get the following bound from above
$$\|h\|_{\infty}\left(6M_v\|\nabla \eta\|_\infty+\int_{\O'}\eta^2|\Delta_\e(\partial_{x_i} v)|\right)$$
exactly as in the proof of \cite[Proposition 7.56]{Ambrosio00}. 

All in all, we get a uniform bound for
$$\int_{\O'}\eta^2|\Delta_\e(\nabla v)|^2$$
when $\e\to 0$. Though we already know that $\Delta_\e(\nabla v)$ converges in the sense of distributions to $\partial_{x_i}(\nabla v)$ so we obtain that it must have a $L^2_{loc}$ representative in ${\O'}$.
\end{proof}
With the previous lemma in hands we can now turn to the proof of Theorem \ref{chapTV:Regularity+_f}:
\newline
\begin{proof}[Proof of Theorem \ref{chapTV:Regularity+_f}]
As will be detailed in the proof of 
Theorem \ref{chapTV:thm_jump_f} (see especially equation (\ref{chapTV:EL_f})) the level sets $\partial E_t$ can be locally represented as the graph of a $C^1$ function $v$ that satisfies the following Euler-Lagrange equation 
\begin{align*}
\diverg_{x'}\left(\nabla_{p'}\Phi(\cdot,v,-\nabla_{x'}v,1)\right)+\partial_{x_N}\Phi(\cdot,v,-\nabla v,1)=\partial_t\Psi\left((\cdot,v),t\right)
\end{align*}
over a ball $B'\subset\R^{N-1}$ and note that we used the notation $x=(x',x_N)$ and $p=(p',p_N)$.\\

This actually means that function $v$ solves in a weak sense
\begin{align*}
-\diverg_{x'}\left(\nabla_{p'}\Phi(\cdot,v,-\nabla_{x'}v,1)\right)=h
\end{align*}
for some $h\in L^\infty(B')$.\\


Now, $\G_p\Phi$ being Lipschitz continuous (see Remark \ref{chapTV:Rmq_def_Phi}), we can apply Lemma \ref{chapTV:lem_A} so $v$ is in $W^{2,2}_{loc}(B')$ and 
we are allowed to expand the divergence. Doing so, we find that $v$ satisfies in a weak sense the following identity
\begin{align*}
-(\diverg_{x'}\G_{p'}\Phi)(\cdot,v,-\G_{x'} v,1)-\nabla_{x'} v\cdot(\partial_{x_N} \G_{p'} \Phi)(\cdot,v,-\G_{x'}v,1)\\
+\ \text{tr}\left({D^2_{x'}v}^TD^2_{p'}\Phi(\cdot,v,-\G_{x'} v,1)\right)=h
\end{align*}
which can be rewritten as
\begin{align*}
\sum_{i,j=1}^{N-1}a_{i,j}\partial_{x_i x_j} v=\tilde{h}
\end{align*}
with
\begin{align*}
a_{i,j}=&\partial_{p_ip_j}\Phi(\cdot,v,-\G_{x'}v,1),\\ 
\tilde{h}=&h+(\diverg_{x'}\G_{p'}\Phi)(\cdot,v,-\G_{x'} v,1)\\
&+\nabla_{x'} v\cdot(\partial_{x_N} \G_{p'} \Phi)(\cdot,v,-\G_{x'}v,1)\in L^\infty(B').
\end{align*}
The $W^{2,p}$ regularity for any $1<p<\infty$ follows from well-known results on the regularity of solutions of elliptic partial differential equations in general form with continuous coefficients (see Theorem \ref{chapTV:thm_reg_edp}). Then Morrey's inequality (see Remark \ref{chapTV:Morrey}) yields the $C^{1,\alpha}$ regularity for any $\alpha<1$. 
\end{proof}
\begin{rmq} Clearly the regularity theorem for elliptic equations in divergence form cannot be applied since the coefficients lack regularity.
This is the reason why we proceeded by first proving Lemma \ref{chapTV:lem_A} to be able to differentiate $\nabla_{x'}v$ and use the second regularity theorem for PDEs in non-divergence form.
\end{rmq} 

\subsection{The discontinuities of solutions of the anisotropic minimum problem} 
We are now ready to state the main result of this section:\\
\begin{thm}\label{chapTV:thm_jump_f}
Let $\Phi$ be a strongly convex reversible Finsler integrand of class $C^2$ on $\O\times\R^N\setminus\{0\}$, $\Psi$ be as in $(H9)$ and that satisfies in addition
for some countable $D$ dense in $\R$
\begin{align*}
\partial_t\Psi(\cdot,t)\in\BV\cap L^\infty(\O)\ \forall t\in D.
\end{align*}
If $u\in\BV$ is the minimizer of (\ref{chapTV:ROF_f}), then one has 
$$ J_u\subset \bigcup_{\substack{t\in D}} J_{\partial_t\Psi(\cdot,t)}$$
up to a $\H^{N-1}$-negligible set.
\end{thm}
\begin{rmq}
\begin{trivlist}
\item[$(i)$]When $\Phi$ does not depend on $x$ and if we set $\Psi(x,t)$ to get the classical quadratic data fidelity term the result was already stated in \cite{ChamJump} and is the key step to get an extension of this theorem when dealing with $TV_\e$. This is not trivial since the latter functional is not positively 1-homogeneous. Let us denote $u_\la$ the minimizer of 
\begin{align*}
\min_{u\in\BV}\la\int_{\O} \sqrt{1+|D{u}|^2}+\frac{1}{2}{\|{u}-{g}\|}_2^2
\end{align*}
where without loss of generality we dropped the $\e$.

The trick is to add another dimension and consider the functions
\begin{align*}
\tilde{u}(x,x_{N+1})&:=u(x)+x_{N+1},\\
\tilde{g}(x,x_{N+1})&:=g(x)+x_{N+1}.
\end{align*}
Then it is possible to prove
that $\tilde{u}_\la$ minimizes (locally)
\begin{align*}
\la\int_{\O\times\R\subset\R^{N+1}} |D\tilde{u}|+\frac{1}{2}{\|\tilde{u}-\tilde{g}\|}_2^2
\end{align*}
thus $J_{u_\la}\times\R=J_{\tilde{u}_\la}\subset J_{\tilde{g}}=J_{g}\times\R$. 
\item[$(ii)$] The ellipticity assumption for $\Phi$ is necessary. Indeed, in \cite{Bellettini99}, it is shown that, in the crystalline case $\Phi(p)={\|p\|}_1$ in dimension $N=2$, it is possible to construct a function $g$ such that one has for the corresponding minimizer $u$, $J_g\not\subset J_u$.\\
\end{trivlist}
\end{rmq}
Our proof follows closely the one given by Caselles, Chambolle and Novaga in \cite{ChamJump}:
\newline
\begin{proof}
Since for any countable set $D$ dense in $\R$
\begin{align*}
J_u\subset\bigcup_{\substack{t_1,t_2\in D\\t_1<t_2}}\partial E_{t_1}\cap \partial E_{t_2},
\end{align*}
it is sufficient to prove that for all $t_1$, $t_2\in D$
\begin{align*}
\partial E_{t_1}\cap \partial E_{t_2}\subset J_{\partial_t\Psi(\cdot,t_1)}\cup J_{\partial_t\Psi(\cdot,t_2)} 
\end{align*}
up to a $\H^{N-1}$-negligible set.\\

To prove the latter inclusion, we are going to reason by contradiction and assume that there are $t_1\not=t_2$ such that
$$\H^{N-1}\left((\partial E_{t_1}\cap \partial E_{t_2})\setminus (J_{\partial_t\Psi(\cdot,t_1)}\cup J_{\partial_t\Psi(\cdot,t_2)})\right)>0.$$
Given that 
$$\H^{N-1}\left((S_{\partial_t\Psi(\cdot,t_1)}\cup S_{\partial_t\Psi(\cdot,t_2)})\setminus (J_{\partial_t\Psi(\cdot,t_1)}\cup J_{\partial_t\Psi(\cdot,t_2)})\right)=0,$$
where we recall that $\O\setminus S_{\partial_t\Psi(\cdot,t_i)}$ is the set of approximate continuity points of $\partial_t\Psi(\cdot,t_i)$ (see Definition \ref{chapTV:def_jump}),  
it is equivalent to assume that
$$\H^{N-1}\left((\partial E_{t_1}\cap \partial E_{t_2})\setminus (S_{\partial_t\Psi(\cdot,t_1)}\cup S_{\partial_t\Psi(\cdot,t_2)})\right)>0.$$
By Theorem \ref{chapTV:Regularity_f}, one can get rid of the closed set where the boundary $\partial E_{t_1}$ and $\partial E_{t_2}$ are not regular and place ourself at a point 
\begin{align*}
\bar{x}\in\partial^* E_{t_1}\cap \partial^* E_{t_2}\setminus (S_{\partial_t\Psi(\cdot,t_1)}\cup S_{\partial_t\Psi(\cdot,t_2)})
\end{align*}
such that both these boundaries can be represented as graphs in the vicinity of $\bar{x}$. That is to say, up to a Euclidian motion, there is a cylindrical neighborhood $\{(x',x_N)\in \R^N\ /\ |x'|<R,\ -R<x_N<R\}$ of $\bar{x}=\left({\bar{x}}',{\bar{x}}_N\right)$ for some small $R>0$ such that $E_{t_i}, i\in\{1,2\}$ coincides with the epigraph of a function $v_i : B({\bar{x}}',R)\rightarrow (-R,R)$ of class $W^{2,q}$, for any $q\geq 1$. Again, throwing away $\H^{N-1}$-negligible sets one can assume that ${\bar{x}}'$ is a Lebesgue point of $v_i$, $\nabla_{x'} v_i$ and $D^2_{x'} v_i$. Actually, by Rademacher-Calder\'on's theorem, we know that $v_i$ and $\nabla_{x'} v_i$ are differentiable a.e. on $B'$ but this stronger result will not be necessary in what follows.\\

By Proposition \ref{chapTV:LevelSetPb_f_thm}, we know that both superlevels $E_{t_i},\ i\in\{1,2\}$ solve the following
$$\min_E \int_{\partial^*E} \Phi(x,\nu_{E})d\mathcal{H}^{N-1}+\int_E \partial_t\Psi(x,t_i)dx$$
where we minimize over all sets of finite perimeter in $\O$.
By doing compact modifications in the ball $B'=B({\bar{x}}',R)$, one can see that the graph $v_i,\ {i\in\{1,2\}}$, minimizes
\begin{align*}
I(v)=\int_{B'} \Phi\left(x',v(x'),\frac{\left(-\nabla_{x'} v(x'),1\right)}{\sqrt{1+|\nabla_{x'} v(x')|^2}}\right)\sqrt{1+|\nabla_{x'} v(x')|^2}\ dx'\\
	+\int_{B'}\int^R_{v(x')}\partial_t\Psi\left((x',x_N),t_i\right) dx_N dx'.
\end{align*}
This means that, for any positive perturbation $\vphi\in C^\infty_c(B')$ of the level set $\partial E_{t_i}$ with $i\in\{1,2\}$,
\begin{align*}
\lim_{\substack{{\e\to 0}\\\e>0}}\frac{I(v_i+\e\vphi)-I(v_i)}{\e}\geq0.
\end{align*}
On the other hand, if we denote $p'=(p_1,...,p_{N-1})$,
\begin{align}\label{chapTV:EL_f} 
I(v_i+\e\vphi)=I(v_i)+\e&\int_{B'}\bigg(\partial_{x_N}\Phi\left(x',v_i(x'),-\G_{x'} v_i(x'),1\right)\vphi(x')\notag\\
&-\G_{p'}\Phi\left(x',v_i(x'),-\G_{x'} v_i(x'),1\right)\cdot\G_{x'}\vphi(x')\notag\\
&-\int^{v_i(x')+\e\vphi(x')}_{v_i(x')}\partial_t\Psi\left((x',x_N),t_i\right)dx_N\bigg)dx'+o(\e).
\end{align}
One should note that the partial differentiations are done in the new set of coordinates.
If we integrate by parts, which is possible by Theorem \ref{chapTV:Regularity+_f}, and given the slicing properties of $BV$ functions (see \cite{Ambrosio00}), one has for $\H^{N-1}$-a.e. $x'\in B'$
\begin{align*}
\partial_{x_N}\Phi ({x},\nu_{E_{t_i}}({x}))+[\diverg_{x'}&\G_{p'}\Phi]({x},\nu_{E_{t_i}}({x}))+\nabla_{x'} v_i({{x}}')\cdot[\partial_{x_N} \G_{p'} \Phi]({x},\nu_{E_{t_i}}({x}))\notag\\
&- D^2_{x'}v_i({\bar{x}}'):D^2_{p'}\Phi({x},\nu_{E_{t_i}}({x}))-\partial_t\Psi\left({x},t_i+0\right)\geq0,
\end{align*}
where $D^2_{x'}v_i({{x}}'):D^2_{p'}\Phi({x},\nu_{E_{t_i}}({x}))$ designates the tensor contraction of the Hessians and is defined as follows
\begin{align*}
D^2_{x'}v_i({{x}}'):D^2_{p'}\Phi({x},\nu_{E_{t_i}}({x}))&=\text{tr}\left({D^2_{x'}v_i}^TD^2_{p'}\Phi({x},\nu_{E_{t_i}}({x}))\right)\\
&=\sum_{k,l=1}^{N-1}\partial_{x_kx_l} v_i({{x}}')\ \partial_{p_kp_l}\Phi({x},\nu_{E_{t_i}}({x})).
\end{align*}
Reasoning in the same way with $\e<0$, one gets
\begin{align*}
\partial_{x_N}\Phi ({x},\nu_{E_{t_i}}({x}))-[\diverg_{x'}&\G_{p'}\Phi]({x},\nu_{E_{t_i}}({x}))-\nabla_{x'} v_i({{x}}')\cdot[\partial_{x_N} \G_{p'} \Phi]({x},\nu_{E_{t_i}}({x}))\notag\\
&+ D^2_{x'}v_i({\bar{x}}'):D^2_{p'}\Phi({x},\nu_{E_{t_i}}({x}))-\partial_t\Psi\left({x},t_i-0\right)\leq0.
\end{align*}
Now, without loss of generality, we can assume $t_2>t_1$. By {Lemma \ref{chapTV:NivCroissants}}, which asserts that $v_2\geq v_1$ a.e. on $B'$, one has
\begin{align*}
v_1({\bar{x}}')&=v_2({\bar{x}}'),\\
\nabla_{x'} v_1({\bar{x}}')&=\nabla_{x'} v_2({\bar{x}}'),\\
D^2_{x'}v_1({\bar{x}}')&\leq D^2_{x'}v_2({\bar{x}}').
\end{align*}
Since in addition we assumed $\bar{x}\not\in S_{\partial_t\Psi(\cdot,t_i)}$ for $i\in\{1,2\}$, we find
\begin{align}\label{chapTV:EL_f_x}
\partial_{x_N}\Phi (\bar{x},\nu_{E_{t_i}}(\bar{x}))+[\diverg_{x'}&\G_{p'}\Phi](\bar{x},\nu_{E_{t_i}}(\bar{x}))+\nabla_{x'} v_i({\bar{x}}')\cdot\partial_{x_N} \G_{p'} \Phi(\bar{x},\nu_{E_{t_i}}(\bar{x}))\notag\\
&- D^2_{x'}v_i({\bar{x}}'):D^2_{p'}\Phi(\bar{x},\nu_{E_{t_i}}(\bar{x}))-\partial_t\Psi\left(\bar{x},t_i\right)=0.
\end{align}
Therefore, substracting the equations (\ref{chapTV:EL_f_x}) we got for the two values of $i$ and using the strict convexity assumption on $\Psi$ (see $(H9)$), we are simply left with
$$D^2_{p'}\Phi(\bar{x},\nu_{E}(\bar{x})):\left(D^2_{x'}v_1({\bar{x}}')-D^2_{x'}v_2({\bar{x}}')\right)=\partial_t\Psi(\bar{x},t_2)-\partial_t\Psi(\bar{x},t_1)>0.$$
Whereas, by non-negativity of $\Phi$ which is asserted by $(H8)$, it follows (see \cite[p. 218]{Lancaster})
$$ D^2_{p'}\Phi(\bar{x},\nu_{E}(\bar{x})):\left(D^2_{x'}v_1({\bar{x}}')-D^2_{x'}v_2({\bar{x}}')\right)\leq 0$$
hence the contradiction.
\end{proof}

\section{Refinement for the weighted Total Variation}
In this section, we are going to focus on the case when the integrand $\Phi$ is simply given by a weight against the total variation measure namely
\begin{align*}
\Phi(x,p)=w(x)|p|.
\end{align*}
For simplicity, we consider that
\begin{align*}
\Psi(x,t)=\frac{1}{2}{\|t-g\|}_2^2.
\end{align*}
This corresponds to the classical quadratic data fidelity term for some Lebesgue measurable $g$. 
Function $\Phi$ is a strongly convex reversible Finsler integrand as soon as \\
\begin{itemize}
\item[$\textbf{(H10)}$] $w:\O\rightarrow\R$ is positive, $\beta$-Hölder with $\beta\in(0,1]$ and there exists a positive real $C_w$ such that
$C_w^{-1}\leq w\leq C_w.$\\
\end{itemize}
Henceforth, $w$ will satisfy this assumption.\\ 

To sum up, from now on, given $w$ that satisfies $(H10)$, we are interested in the minimizer $u$ of the following problem
\begin{align}\label{chapTV:ROF_w}
\min_{u\in\BV}\int_\O w|Du|+\frac{1}{2}{\|u-g\|}^2_2.
\end{align}
Its superlevels $E_t=\{u>t\}$ solve the minimal surface problem
\begin{align}\label{chapTV:LevelSetPb_w}
\min_E P_w(E,\O)+\int_E (t-g(x))dx
\end{align}
over sets of finite perimeter in $\O$, where we recall that 
$$P_w(E,\O)=\int_{\partial^*E\cap\O}w(x)d\H^{N-1}(x)$$
is the weighted perimeter.\\

All the results we developed for general Finsler integrands are still valid in this special case. In particular, the anisotropic coarea formula and the regularity Theorem \ref{chapTV:Regularity_f} for quasi minimizers of the perimeter hold true.

%

\subsection{More regularity in the weighted case}
\label{chapTV:weighted_case}

As already seen in the proof of Theorem \ref{chapTV:thm_jump_f}, it is important to be able to say that the level sets of minimizers of problems involving the total variation are regular namely $C^{1,\alpha}$ for some $\alpha\in(0,1/2)$. Such results stem from the theory of regularity of minimal surfaces and have become classical in the literature. We already mentioned the works \cite{Almgren77,Schoen,Bombieri,Duzaar} that deal with minimizers or quasi-minimizers of the perimeter in the anisotropic setting. In particular, they establish the regularity of minimizers of the perimeter that have a prescribed curvature.\\

Similar regularity results have also been established for various constraints. 
Let us mention two recent examples.
In \cite{Morgan03}, Frank Morgan proved such a regularity for isoperimetric surfaces. Indeed, he shows that an isoperimetric hypersurface of dimension at most six in a smooth Riemannian manifold is a smooth submanifold. If the metric is merely Lipschitz, then it is still $C^{1,\alpha}$ for any $\alpha>1$.\\

In the recent article \cite{Figalli}, Figalli and Maggi are led to consider a problem with both a constraint on the curvature and on the volume \emph{i.e.}
\begin{align}\label{chapTV:Pb_Figalli}
\min_E\left\{\int_{\partial^*E} \Phi(\nu_E)d\H^{N-1}+\int_E g\ /\ |E|=m\right\}
\end{align}
for a positively 1-homogeneous elliptic $\Phi$, coercive $g$ and for small $m$.
In Appendix C of their work, they discuss the $C^{1,\alpha}$ regularity of a solution of (\ref{chapTV:Pb_Figalli}) as a consequence of the regularity theory for quasi-minimizers of the anisotropic perimeter \cite{Duzaar}.\\

The papers we just cited make a wide usage of the language of currents and varifolds. Some other works (\cite{Massari74,Massari75,MassariPepe75,Tamanini82,Tamanini,Ambrosio99}) are based on techniques that date back to the results of De Giorgi and deal with these regularity issues in the framework of sets of finite perimeter. 
The proof we provide is simple and does not rely on the theory of currents. Our result is partially contained in the above-mentioned Theorem \ref{chapTV:Regularity_f} for the anisotropic total variation which follows from \cite{Duzaar} but to our knowledge there is no simple proof of it in the context of sets of finite perimeter.\\

First, let us recall the concept of quasi-minimizer:
\newline
\begin{deff}
Let $E$ be a set of finite perimeter in $\O$, $w$ satisfy assumption $(H10)$, $\alpha\in(0,1)$ and $\Lambda\geq0$. Then $E$ is a \textit{$(\Lambda,\alpha)$-quasi-minimizer of the perimeter $P_w$ in $\O$} or simply \textit{quasi-minimizer} if
\begin{align}\label{chapTV:quasi_min_def}
P_w\left(E,B(x,r)\right)\leq P_w\left(F,B(x,r)\right)+\Lambda {|E\Delta F|}^{1+\frac{2\alpha-1}{N}}
\end{align}
for any ball $B(x,r)\subset\subset\O$ with $r>0$ and any $F\subset\O$ of finite perimeter such that $F\Delta E\subset\subset B(x,r)$.
\end{deff}
\begin{rmq}\label{chapTV:def_Tam}
We could have replaced (\ref{chapTV:quasi_min_def}) by the weaker condition
\begin{align}\label{chapTV:quasi_min_def_bis}
P_w\left(E,B(x,r)\right)\leq P_w\left(F,B(x,r)\right)+\Lambda r^{N-1+2\alpha}
\end{align}
but for simplicity we refer to \cite{Tamanini82,Tamanini} where the author considers this definition.
\end{rmq}
The aim is to show that the following regularity for quasi-minimizers of the weighted perimeter holds:
\newline
\begin{thm}\label{chapTV:quasi_min}
Let $\O$ be an open set of $\R^N$, $N\geq2$, $w:\O\to\R\ \beta$-Hölder for some $\beta\in(0,1]$ and such that there is a positive real $C_w$ with $C_w^{-1}\leq w\leq C_w$. Consider also $\alpha\in(0,\frac{1}{2}),\ \Lambda\geq0$ and $E$ a set of finite perimeter that is a $(\Lambda,\alpha)$-minimizer of the perimeter $P_w$.\\\\
Then, if we denote $\gamma=\min(\alpha,\frac{\beta}{2})$, the reduced boundary $\partial^* E$ is a $C^{1,{\gamma}}$-hypersurface and
$$\H^s(\partial E\setminus\partial^*E)=0$$
for every $s>N-8$.\\\\
Moreover, let us assume that ${(E_h)}_{h\in\R}$ are $(\Lambda,\alpha)$-minimizers of the perimeter $P_w$ with $E_h$ converging locally to some limit set $E_\infty$ as $h\to+\infty$. If $x_h\in\partial E_h$ for every $h$ and converges as $h\to+\infty$ to some $x_\infty\in\O$ then $x_\infty\in\partial E_\infty$. If, in addition, $x_\infty\in\partial^*E_\infty$ then there exists $h_0$ such that for $h\geq h_0$, $x_h\in\partial^*E_h$ and the unit outward normal to $\partial^*E_h$ at $x_h$ converges to the unit outward normal to $\partial^*E_\infty$ at $x_\infty$.\\
\end{thm}


The theorem is well-known for quasi-minimizers of the classical perimeter (even with the weaker condition (\ref{chapTV:quasi_min_def_bis})) and follows from \cite{Tamanini} whose work is based on earlier papers of Massari (\cite{Massari74,Massari75,MassariPepe75}). Thus, to get the announced regularity, it is sufficient to prove that a quasi-minimizer of $P_w$ is a quasi-minimizer of the classical perimeter $P$. The argument is based on a key ingredient: the density lemma. The latter result is well-known for problems involving the perimeter.
The density lemma also plays an important role in \cite{JalalzaiStairc} to prove that minimizers of ROF have large flat zones. 
\newline
\begin{lem}[Density estimate]\label{chapTV:density}
Let $w$ be as in assumption $(H10)$, $\alpha\in(0,1)$, $\Lambda\geq0$ and consider $E$ a set of finite perimeter that is a \textit{$(\Lambda,\alpha)$-quasi-minimizer of $P_w$}. Then there exists a radius $r_0>0$ and $C>0$ depending only on $N$ and $w$ such that for any point $x\in\O$,\\
\begin{itemize}
\item[-] if $\forall r>0,\ |E\cap B(x,r)|>0 \text{ then } \forall r<r_0,\ |E\cap B(x,r)|\geq \frac{w_{N}r^{N}}{2^NC_w^{N}},$
\item[-] if $\forall r>0,\ |B(x,r)\setminus E |>0 \text{ then } \forall r<r_0,\ |B(x,r)\setminus E|\geq \frac{w_{N}r^{N}}{2^NC_w^{N}}$.\\
\end{itemize}

\noindent In particular, if $x\in\partial^* E$,
\begin{align*}
\forall r<r_0,\ \min{\left(|E\cap B(x,r)|,|B(x,r)\setminus E|\right)} \geq \frac{w_{N}r^{N}}{2^NC_w^{N}}.
\end{align*}
Moreover, one has for the classical perimeter
\begin{align*}
C^{-1}r^{N-1}\leq P(E,B(x,r))\leq Cr^{N-1}.
\end{align*}

\begin{rmq}
\begin{trivlist}
\item[$(i)$]
The assertion on the perimeter is sometimes referred to as the \textit{Ahlfors regularity} of the boundary $\partial E$. 
\item[$(ii)$] A variant of this lemma holds also for the anisotropic perimeter $P_\Phi$ (see \cite{Caffarelli} for instance).\newline
\end{trivlist}
\end{rmq}
\end{lem}

In our problem, the key point is that (\ref{chapTV:quasi_min_def}) can be rewritten in terms of the classical perimeter in the following way:
\begin{align*}
P_w(E,B(x,r))=w(x)P(E,B(x,r))+\int_{\partial^*E\cap B(x,r)}(w(y)-w(x))d\H^{N-1}\\
\leq w(x)P(F,B(x,r))+\int_{\partial^*F\cap B(x,r)}(w(y)-w(x))d\H^{N-1}
+\Lambda{|E\Delta F|}^{1+\frac{2\alpha-1}{N}}
\end{align*}
But, since $w$ is $\beta$-Hölder, the quasi-minimality condition becomes
\begin{align*}
(w(x)-{\|w\|}_{C^{0,\beta}} r^\beta)P(E,B(x,r))\leq (w(x)+{\|w\|}_{C^{0,\beta}} r^\beta) P(F,B(x,r))\\+\Lambda{|E\Delta F|}^{1+\frac{2\alpha-1}{N}}.
\end{align*}
To alleviate notations, we are going to assume that $w$ has Hölder norm ${\|w\|}_{C^{0,\beta}}=1$ 
and we are going to write $B_r$ for the ball $B(x,r)$. So for some small radius $r$ with $r^\beta<C_w^{-1}$,
 we are simply left with
\begin{align}\label{chapTV:quasi_min_simple}
{(w(x)-r^\beta)}P(E,B_r)\leq {(w(x)+r^\beta)} P(F,B_r)+\Lambda {|E\Delta F|}^{1+\frac{2\alpha-1}{N}}.
\end{align}
Having this remark in mind, we can now get to the proof:
\newline
\begin{proof}
Let us prove the first item of the lemma. The idea is to compare the energy of $E$ with that of $E\setminus B_r$. 
Let $f(r)=|E\cap B_r|>0$ for all $r>0$. We see that it is a non decreasing function thus differentiable almost everywhere and by the coarea formula one knows that
$f'(r)=\H^{N-1}(E\cap\partial B_r)$ for a.e. $r$.
Now by the isoperimetric inequality,
\begin{align*}
N\omega_N^{\frac{1}{N}}f(r)^{\frac{N-1}{N}}\leq P(E\cap B_r,\R^N)=\H^{N-1}(\partial^*E\cap B_r)+\H^{N-1}(E\cap \partial B_r).
\end{align*}
But by minimality of $E$,
\begin{align*}
\H^{N-1}(\partial^*E\cap B_r)\leq \frac{w(x)+r^\beta}{w(x)-r^\beta} \H^{N-1}(E \cap \partial B_r)+\Lambda \frac{{f(r)}^{1+\frac{2\alpha-1}{N}}}{w(x)-r^\beta}.
\end{align*}
So
\begin{align*}
\left(N\omega_N^{\frac{1}{N}}-\frac{\Lambda {f(r)}^\frac{2\alpha}{N}}{{w(x)-r^\beta}}\right) f(r)^{\frac{N-1}{N}}
&\leq \left(1+\frac{w(x)+r^\beta}{w(x)-r^\beta}\right)f'(r),
\end{align*}
which implies
\begin{align*}
\frac{w(x)-r^\beta}{2}\left(\omega_N^{\frac{1}{N}}-\frac{\Lambda {f(r)}^\frac{2\alpha}{N}}{N(w(x)-r^\beta)}\right)\leq{\left(f(r)^{\frac{1}{N}}\right)}^\prime.
\end{align*}
Now for $\e\in(0,1)$, we can find $r_\e>0$ such that for a.e. $r<r_\e$,
\begin{align*}
\frac{w(x)-{r_\e}^\beta}{2}\left(\omega_N^{\frac{1}{N}}-\frac{\Lambda \e}{N(w(x)-{r_\e}^\beta)}\right)\leq{\left(f(r)^{\frac{1}{N}}\right)}^\prime.
\end{align*}
Integrating between $0$ and $r_\e$ and sending $\e\to0$, one obtains
\begin{align*}
\frac{w(x)^{N}w_N}{2^N}\leq\liminf_{r\to 0} \frac{f(r)}{r^N}
\end{align*}
hence the first assertion of the lemma.\\

Reasoning in a similar way with $f(r)=|B_r \setminus E|$ one proves the second item of the lemma.\\

Let us prove the last statement. As above by comparison of $E$ with $E\setminus B_r$,
\begin{align*}
(w(x)-r^\beta)P(E,B_r)\leq (w(x)+r^\beta)\H^{N-1}(E \cap \partial B_r)+\Lambda {|E\cap B_r|}^{1-\frac{1}{N}}.
\end{align*}
Thus we obtain
\begin{align*}
P(E,B_r)\leq C\left(\Lambda+\frac{C_w+r^\beta}{C_w^{-1}-r^\beta} \right)r^{N-1}
\end{align*}
for some constant $C$ that only depends on $N$.
\newline\newline
In the last statement, the inequality to the left is obtained by applying the relative isoperimetric inequality.
\end{proof}
With this lemma in hands, we are in a position to prove that a quasi-minimizer of the weighted perimeter is simply a quasi-minimizer of the perimeter. Take $E$ a quasi-minimizer of $P_w$ as in Definition \ref{chapTV:quasi_min_def}. Then, using the notations introduced before the proof of the lemma, we have from (\ref{chapTV:quasi_min_simple}) that for any admissible $F$,
\begin{align*}
 P(F,B_r)+\Lambda \frac{{|E\Delta F|}^{1+\frac{2\alpha-1}{N}}}{w(x)+r^\beta}&\geq \frac{w(x)-r^\beta}{w(x)+r^\beta}P(E,B_r)\\
&\geq (1-2C_w r^\beta)P(E,B_r)\\
&\geq P(E,B_r)-2CC_w r^{N-1+\beta},
\end{align*}
where in the last line we used the density lemma.
Then
\begin{align*}
 P(E,B_r)&\leq P(F,B_r)+r^{N-1}({\Lambda C_w}r^{2\alpha}+2CC_w r^\beta)\\
&\leq P(F,B_r)+C'r^{N-1+2\gamma},
\end{align*}
for some positive constant $C'$ and Theorem \ref{chapTV:quasi_min} follows from the regularity of quasi-minimizers for the classical perimeter (see Remark \ref{chapTV:def_Tam} and also \cite{Tamanini82,Tamanini}). \\

Consider now $u$ a minimizer of (\ref{chapTV:ROF_w}) for $g\in L^p(\O)$ with $p>N$ and let $E_t=\{u>t\}$ for some $t\in\R$. Then, if $x\in\O,\ r>0$ and $F$ is a compact modification of $E_t$ in $B(x,r)$ \textit{i.e.}
$F\Delta E_t\subset\subset B(x,r)$, one has
\begin{align*}
P_w(E_t,B(x,r))&\leq P_w(F,B(x,r))+\int_{E_t\Delta F} |t-g|\\
&\leq P_w(F,B(x,r))+|{E_t\Delta F}|^{1-\frac{1}{p}}\|t-g\|_{L^p(B(x,r))}\\
&\leq P_w(F,B(x,r))+\Lambda r^{N\left(1-\frac{1}{p}\right)},
\end{align*}
so the superlevel $E_t$ is a quasi-minimizer of $P_w$ (and satisfies also the weaker definition (\ref{chapTV:quasi_min_def_bis})). Therefore, Theorem \ref{chapTV:quasi_min} applies with $\alpha=\frac{1}{2}(1-\frac{N}{p})$. In the end, we get exactly the same regularity as in Theorem $\ref{chapTV:Regularity_f}$ but this time we also know that the singular set has dimension at most $N-8$.

Note that there is no way to adapt our approach in the anisotropic case since this would contradict the counterexample of Frank Morgan (see the remark that follows Theorem \ref{chapTV:Regularity_f}). Thus, a quasi-minimizer of the anisotropic perimeter is not necessarily a quasi-minimizer of the classical perimeter.

Now if one assumes that $g\in L^\infty(\O)$, we recall that the Nirenberg's method and the regularity theory for elliptic PDEs let us gain a little regularity as can be seen from the following reformulation of Theorem \ref{chapTV:Regularity+_f}:\\
\begin{thm}\label{chapTV:quasi_min_fort}
Let $\O$ be an open set of $\R^N$, $N\geq2$, $w:\O\to\R$ be Lipschitz continuous and such that there exists a positive real $C_w$ with $C_w^{-1}\leq w\leq C_w$. Consider also $E$ a set of finite perimeter that is a quasi-minimizer of the perimeter $P_w$. 
Then $\partial^* E$ is $W^{2,p}$ for all $1<p<\infty$ and thus $C^{1,\gamma}$ for any $\gamma<1$ and 
$$\H^s(\partial E\setminus\partial^*E)=0$$
for every $s>N-8$.\\ 
\end{thm}
\begin{rmq}
As was done before (see (\ref{chapTV:EL_f})), the level set $\partial E_t$ can be locally represented as the graph of a $C^1$ function $v$ that satisfies the following Euler-Lagrange
\begin{align*}
-\diverg_{x'}\left(w(x',v(x'))\frac{\nabla_{x'}v(x')}{\sqrt{1+{|\nabla_{x'}v(x')|}^2}}\right)+{\partial_{x_N}}w(x',v(x'))\sqrt{1+{|\nabla_{x'}v(x')|}^2 }\\
			=\left(t-g(x',v(x'))\right)
\end{align*}
over a ball $B'\subset\R^{N-1}$. We recall that we denoted $x=(x',x_N)\in\R^N$.
It follows that function $v$ solves 
\begin{align*}
-\diverg_{x'}\left(w(\cdot,v)\frac{\nabla_{x'}v}{\sqrt{1+{|\nabla_{x'}v|}^2}}\right)=h
\end{align*}
with $h\in L^\infty(B')$.
In particular when $N=2$, Rademacher's theorem implies that $w(\cdot,v){v'}/{\sqrt{1+|v'|^2}}$ is Lipschitz continuous which in turn implies that $v$ is locally of class $C^{1,1}$.
This provides additional regularity for the weighted total variation in dimension 2.
\end{rmq}

\subsection{Discontinuities for the adaptive total variation minimization problem}
In the weighted case, we can now get a refinement of the jump inclusion result.
\begin{thm}\label{chapTV:thm_jump_w}
Let $w:\O\rightarrow\R$ be positive, bounded, Lipschitz continuous with $\nabla w\in BV(\O,\R^N)$ and $g\in\BV\cap L^\infty(\O)$. Then, denoting $J_{\G w}:=\bigcup_{i=1}^N J_{\partial_{x_i}w}$, the minimizer $u\in\BV$ of (\ref{chapTV:ROF_w}) satisfies 
\begin{align}\label{chapTV:jump_inclusion}
J_u\subset J_g\cup J_{\G w} 
\end{align}
up to a $\H^{N-1}$-negligible set.\\
Moreover, we have the following bound on the mean curvature of the jump set
\begin{align*}
\kappa_{J_u}\in\left[C_w^{-1}(g^--u^-),C_w(g^+-u^+)\right]\  \H^{N-1}\text{-a.e.}
\end{align*}
If in addition we assume that $w$ is of class $C^1$ we get that at the discontinuity
\begin{align}\label{chapTV:cont_decrease}
(u^+-u^-)\leq(g^+-g^-)\  \H^{N-1}\text{-a.e. on } J_u.
\end{align}
\end{thm}

\begin{rmq}\label{chapTV:spaceBH}
Assumption $\G w\in BV(\O,\R^N)$ means that $w$ lies in the space $BH(\O)$ of bounded Hessian functions that has been thouroughly studied by Demengel in \cite{Demengel84}. It is possible to obtain results that are similar in nature to those known for the $BV$ space. Let us mention in particular that these functions have a $W^{1,1}$ trace, there is also an extension theorem, a Poincaré type inequality and a continuous inclusion $BH(\O)\subset C^0(\bar{\O})$ for Lipschitz domains of $\R^2$ (see also \cite{Savare}).\\
\end{rmq}

Let us make few comments before getting to the proof. First of all, it is interesting to see that whenever the weight is merely Lipschitz continuous it is possible to add discontinuities to the minimizer that were not contained in the datum $g$. To illustrate this point, we give few numerical experiments in dimension one:
\vspace{-0.3cm}
\begin{figure}[H]
\begin{minipage}[c]{.49\linewidth}
\includegraphics[width=\linewidth]{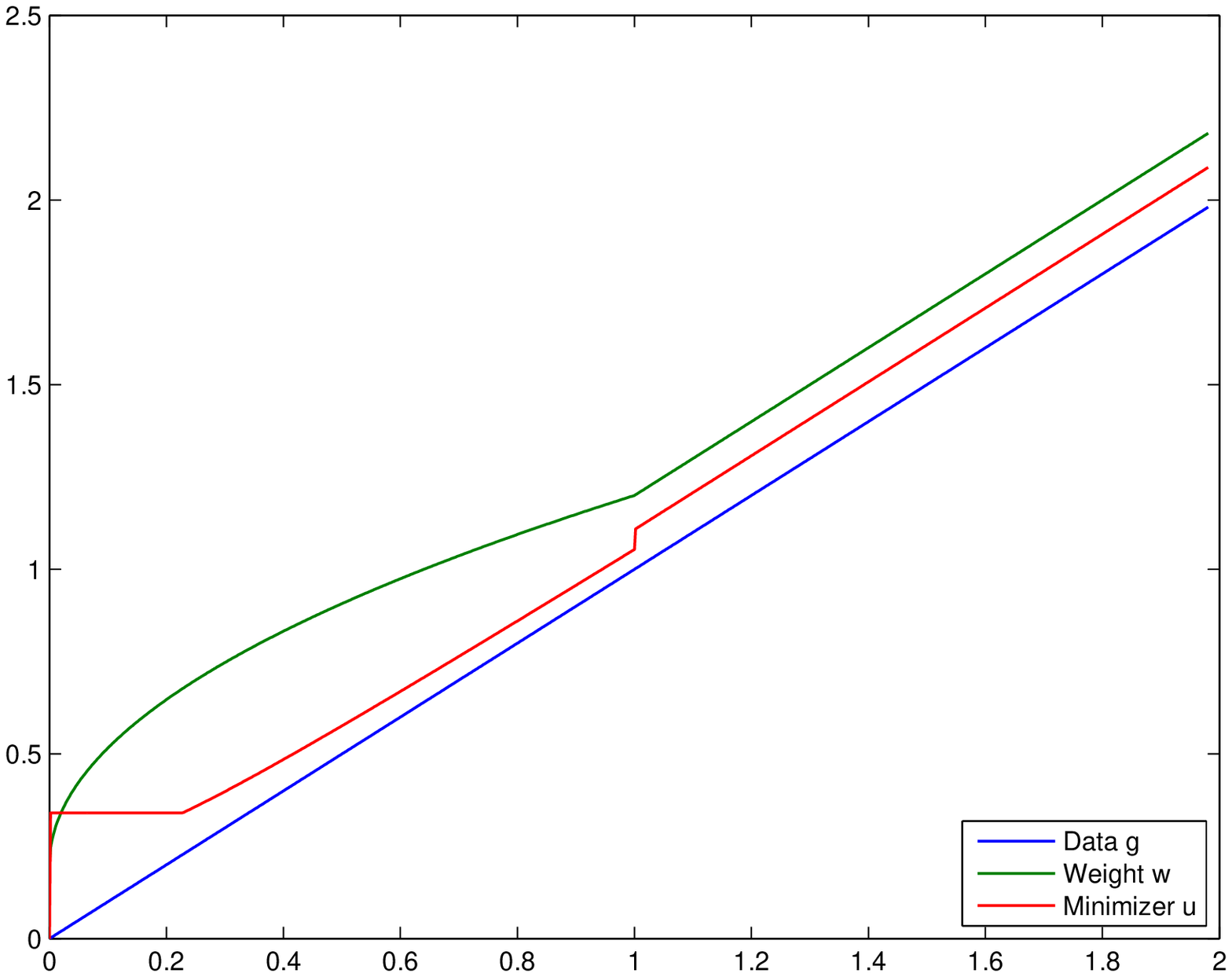}
\vspace{-0.6cm}    
\caption
{Creation of jumps with 
$w(x)=\sqrt{x}\chi_{\{x\leq1\}}+ x\chi_{\{x>1\}}+0.2$}
\end{minipage} 
\hfill
\begin{minipage}[c]{.49\linewidth}
\includegraphics[width=\linewidth]{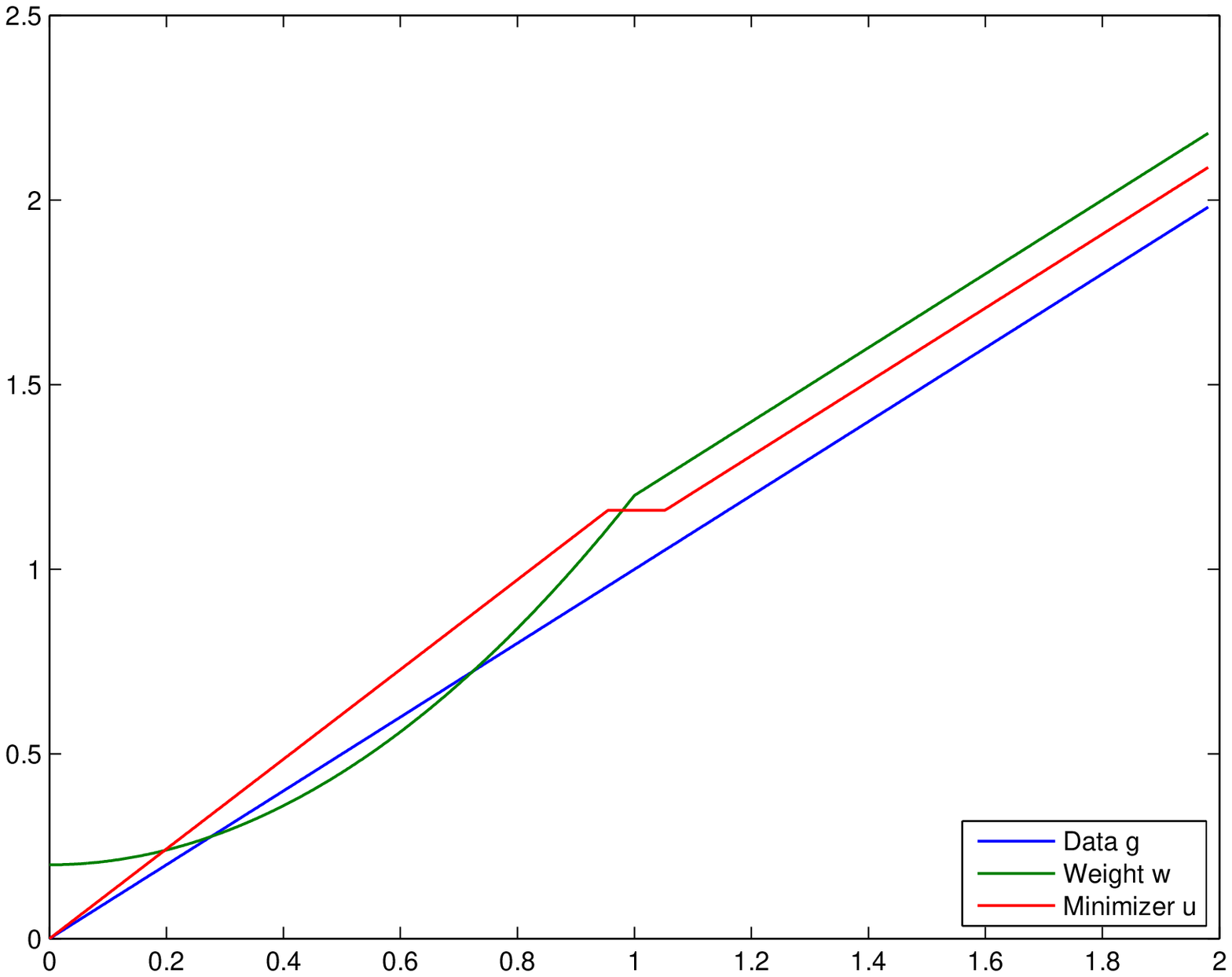}
\vspace{-0.6cm}   
 \caption
{Creation of a flat zone for
$w(x)={x}^2\chi_{\{x\leq1\}}+ x\chi_{\{x>1\}}+0.2$}
\end{minipage}
\end{figure}
\vspace{-0.5cm}
\begin{figure}[H]
\begin{minipage}[c]{.49\linewidth}
\includegraphics[width=\linewidth]{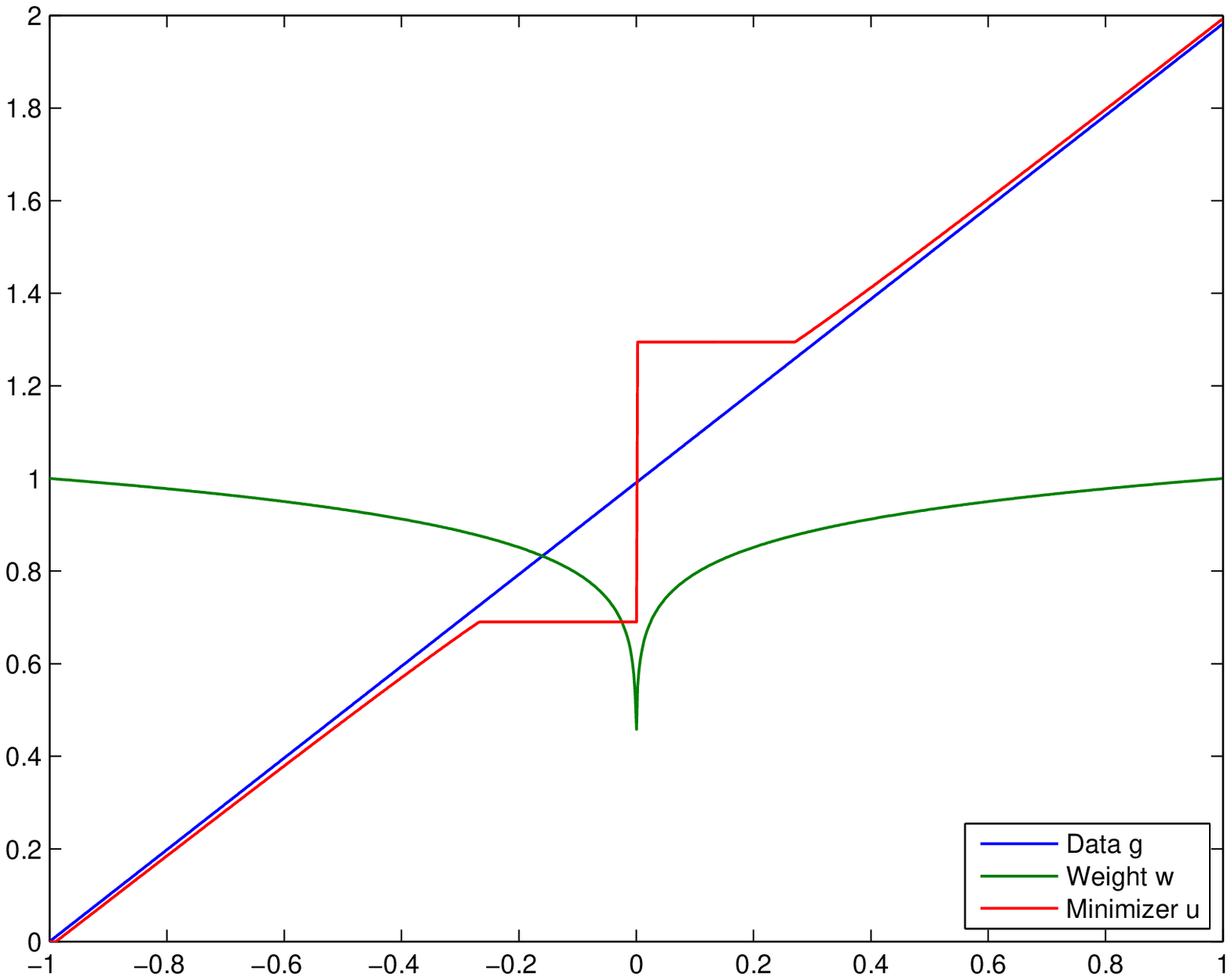}
\vspace{-0.6cm}    
\caption
{Hölder weight $w(x)={|x|}^{1/10}$}
\end{minipage}
\hfill
\begin{minipage}[c]{.49\linewidth}
\includegraphics[width=\linewidth]{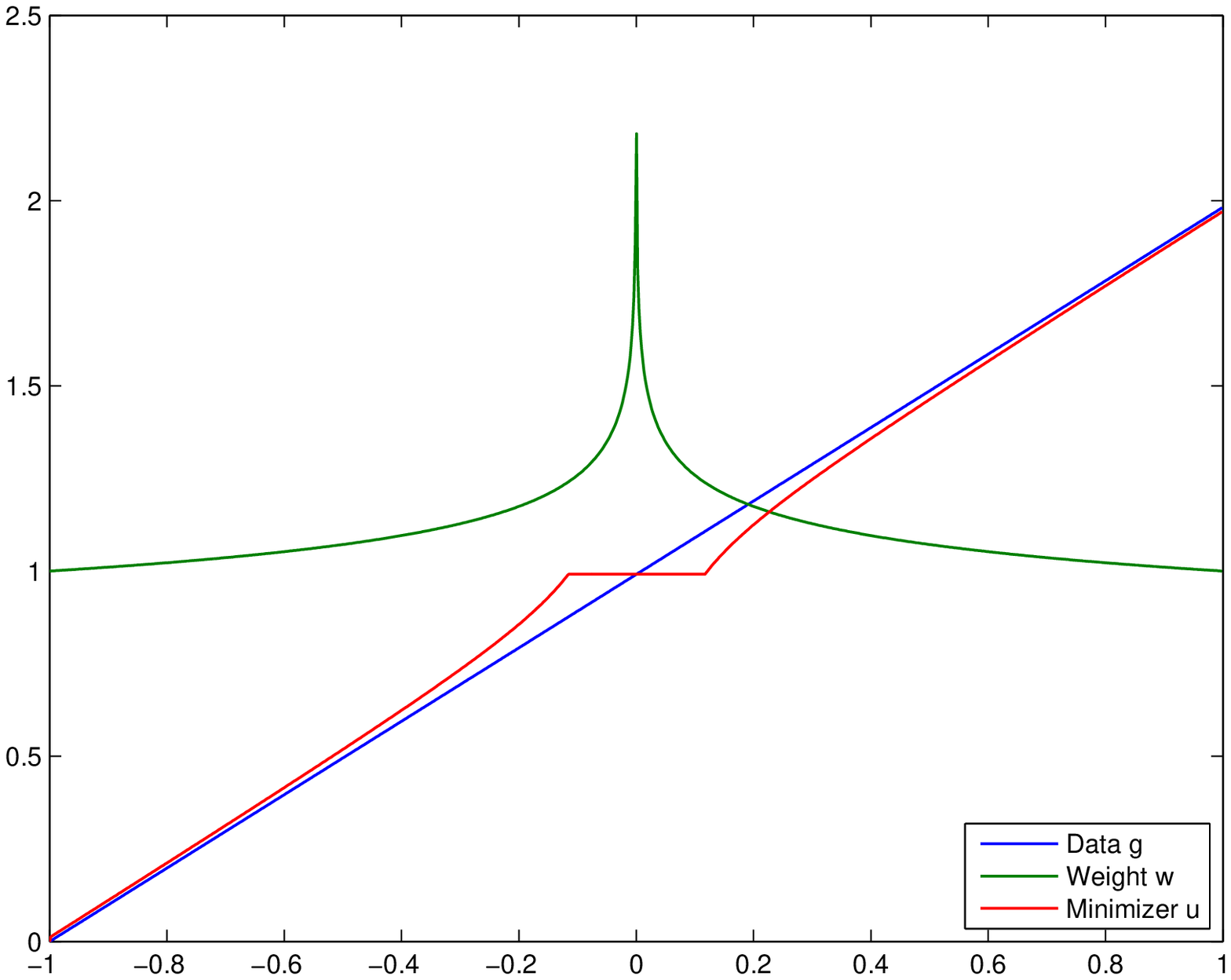}
\vspace{-0.6cm}
\caption
{Weight function $w(x)={|x|}^{-1/10}$}
\end{minipage} 
\end{figure}

These simulations suggest that tweaking $w$ so that $w'$ contains a discontinuity one can force the creation of jumps for some smooth $g$. 

In case the weight function $w$ is of class $C^1$, $J_{\G w}=\emptyset$ which implies $J_u\subset J_g$.

For such a smooth $w$, (\ref{chapTV:cont_decrease}) means that the ``contrast'' (if one thinks of images) decreases at the discontinuity set $J_u$. This is not that surprising for natural images but quite counterintuitive if we consider the following function
$$
\begin{array}[t]{lrcl}
g : & {[0,2\pi )}^2 & \longrightarrow & \R \\
    & (x,y) & \longmapsto & 
\begin{cases} 
2+\cos(x) \text{\ if\ } y>0,\\
0 \text{\ otherwise.}
\end{cases}
\end{array}
$$
\noindent provided periodic boundary conditions.  Let us illustrate this example by a numerical experiment to get a clear idea of what is going on:
we minimize $ROF$ functional with the data function $g$ that is above.
\vspace{-0.5cm}
\begin{figure}[H]   \begin{minipage}[c]{.49\linewidth}
     \includegraphics[width=\linewidth]{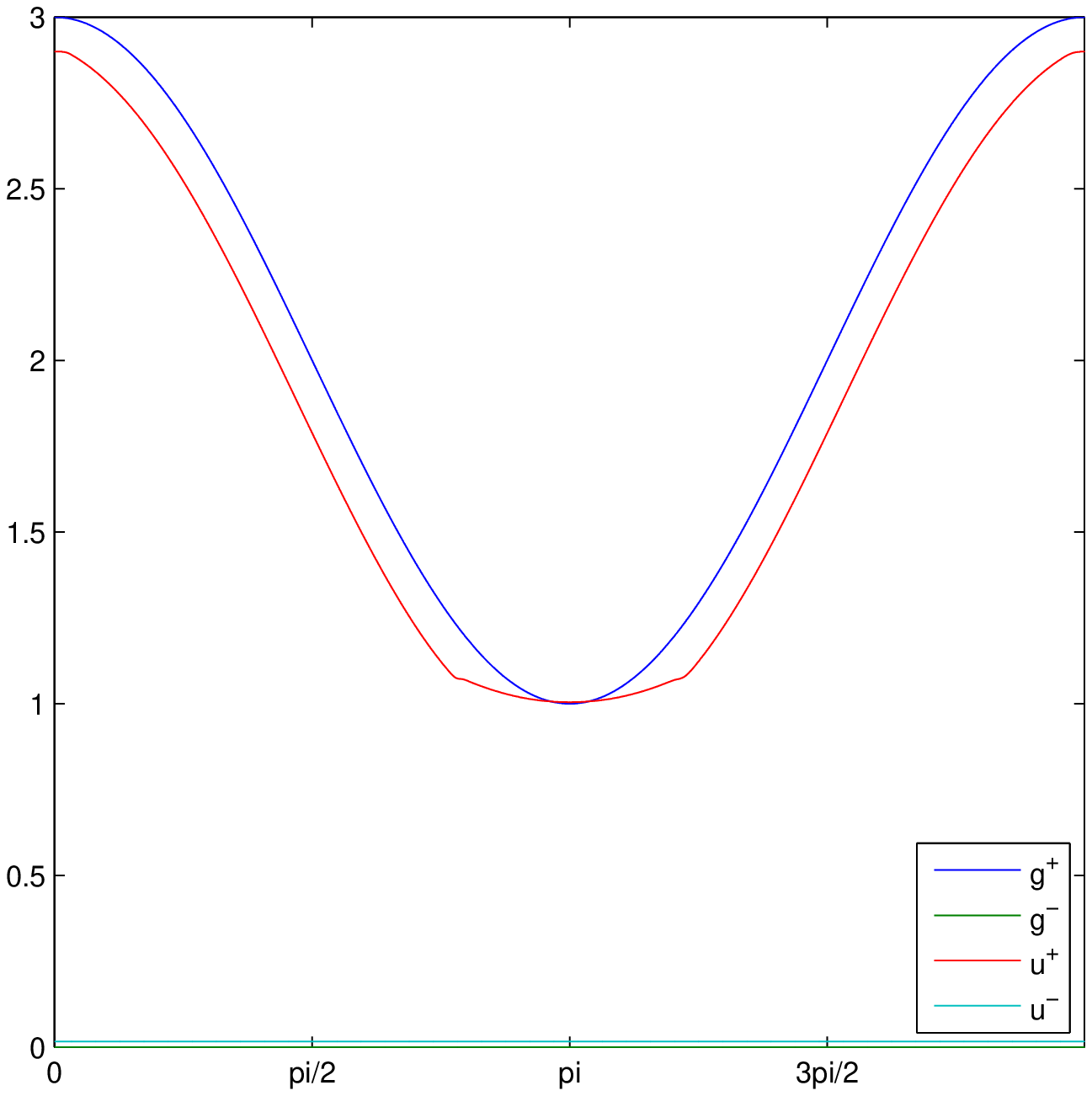}
\vspace{-0.6cm}
\caption
{$u$ and $g$ at the jump set. }
  \end{minipage}
\hfill
  \begin{minipage}[c]{.49\linewidth}
\includegraphics[width=\linewidth]{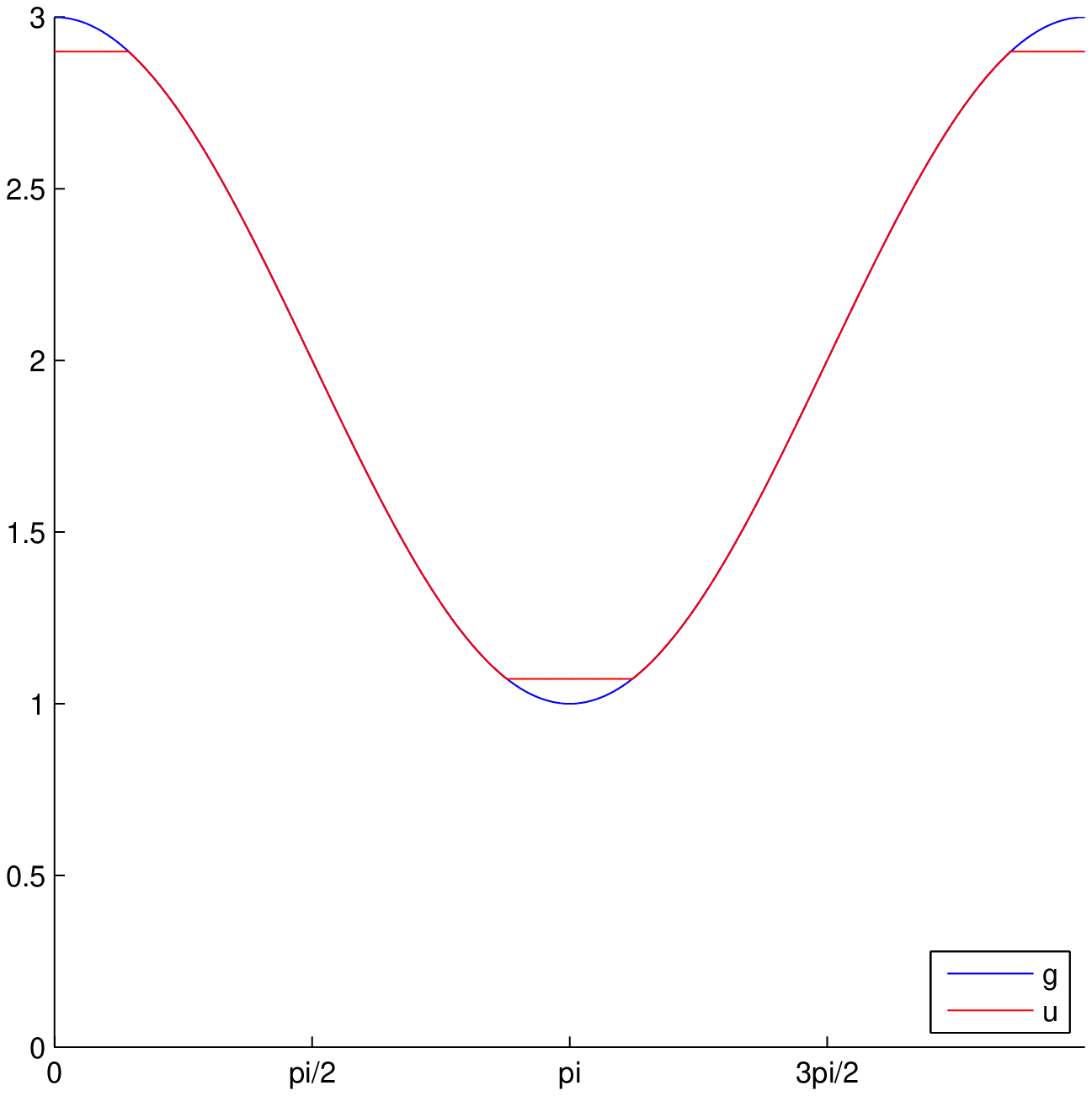}
\vspace{-0.6cm}
\caption
{$u$ and $g$ far from the discontinuity.}
  \end{minipage} 
\end{figure}
\vspace{-0.5cm}
\begin{figure}[H]
  \begin{minipage}[c]{.49\linewidth}
\includegraphics[width=\linewidth]{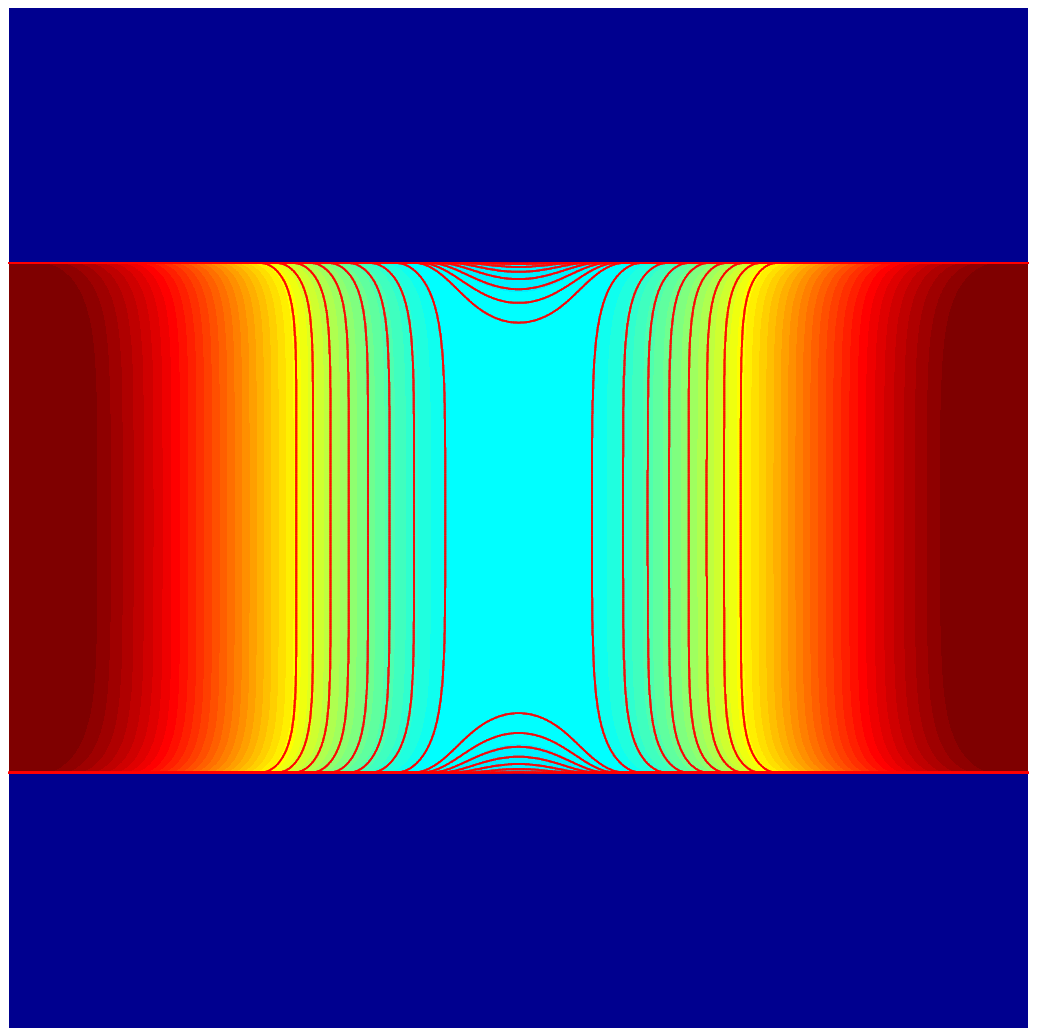}
\vspace{-0.6cm}
\caption
{Level lines $\{u=t\}$ for some values of $t\in(1,2)$. }
  \end{minipage} 
\hfill
  \begin{minipage}[c]{.49\linewidth}
\includegraphics[width=\linewidth]{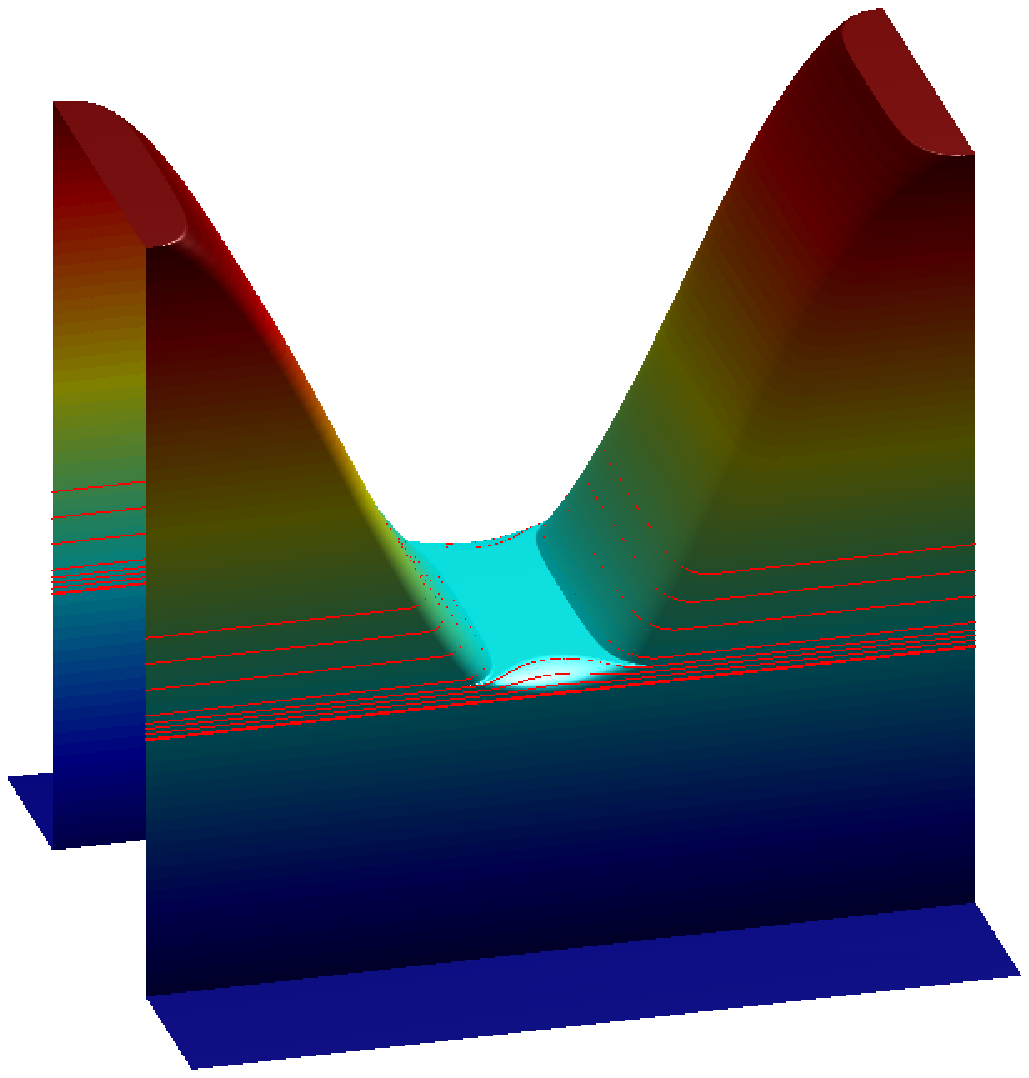}
\vspace{-0.6cm}
\caption
{Graph of $u$ on one period. Some level lines are represented in red.}
  \end{minipage}
 \end{figure}

One can clearly see that little bumps are created near the discontinuities to keep the jump as small as possible. We recall that, far from the jump set, we expect the solution to be constant on large neighborhoods of the extrema and to have a lower infinity norm (see \cite[Thorems 3.1 and 3.3]{JalalzaiStairc}).\\


The proof that follows is slightly different from the one given for Theorem \ref{chapTV:thm_jump_f} since this time we are no longer going to reason by contradiction. This way we can get the desired refinement in the weighted case.\\

\begin{proof}
We recall that up to a $\H^{N-1}$-negligible set
\begin{align*}
J_u\subset\bigcup_{\substack{t_1,t_2\in D\\t_1<t_2}}\partial E_{t_1}\cap \partial E_{t_2}
\end{align*}
for any countable $D$ dense in $\R$, thus it is enough that the result for any $t_1,t_2\in D$ and for $\H^{N-1}$-any $\bar{x}\in\partial E_{t_1}\cap\partial E_{t_2}$.
Combining Theorems \ref{chapTV:quasi_min} and \ref{chapTV:quasi_min_fort} one can assume that both these boundaries can be represented by smooth graphs near $\H^{N-1}$-every $\bar{x}$. That is to say that up to a Euclidian motion there exists a cylinder  $\{x=(x',x_N)\in \R^N\ /\ |x'|<R,\ -R<x_N<R\}$ neighborhood of $\bar{x}$ such that $E_{t_i}$, $i\in\{1,2\}$, coincides with the epigraph of a function  $v_i:B'=B(\bar{x}',R)\to(-R,R)$ of class $W^{2,q}$ for any $q\geq 1$.
We also assume that we have
\begin{align*}
\H^{N-1}\big(\{x'\in B'\ /\ v_1(x')=v_2(x')\}\big)>0
\end{align*}
for the contact set.
Without loss of generality one can finally suppose that $t_2>t_1$ which implies by Lemma \ref{chapTV:NivCroissants} that $v_2\geq v_1$ a.e. on $B'$. Moreover, $\H^{N-1}$-every $x'\in B'$ is a Lebesgue point of functions $v_i$, $\nabla v_i$, $D^2 v_i$, $i\in\{1,2\}$ thus at $\H^{N-1}$-almost every contact point one has
\begin{align}\label{chapTV:inclusion_courb}
v_1({x}')&=v_2({x}'),\notag\\
\nabla_{x'} v_1({x}')&=\nabla_{x'} v_2({x}'),\notag\\
D^2_{x'}v_1({x}')&\leq D^2_{x'}v_2({x}').
\end{align}
Recall that Proposition \ref{chapTV:LevelSetPb_f_thm} tells us that the superlevels $E_{t_i}$ with $i\in\{1,2\}$ solve
\begin{align*}
\min_E \int_{\partial^*E}w(x)d\H^{N-1}(x)+\int_E (t_i-g(x))dx
\end{align*}
 where the minimization is carried out on all sets of finite perimeter in $\O$. Doing compact modifications in the ball $B'$ one immediately sees that $v_i$, $i\in\{1,2\}$ minimizes
\begin{align}\label{chapTV:deriv_I}
I(v)=\int_{B'}w(x',v_i(x'))\sqrt{1+{|\G_{x'} v_i(x')|}^2}dx'+\int_{B'}\int_{v(x')}^R t_i-g(x',x_N)dx_Ndx'.
\end{align}
This means that for any perturbation $\vphi\in C^\infty_c(B')$ such that $\vphi\geq0$
\begin{align*}
{I'(v_i)}^+\cdot\vphi=\lim_{\substack{\e\to 0\\ \e>0}} \frac{I(v_i+\e\vphi)-I(v_i)}{\e}\geq 0
\end{align*}
whereas we know by the slicing properties of $BV$ functions (see in particular \cite[Remark 3.109]{Ambrosio00}) that for a.e. $x'\in B'$ 
\begin{align*}
\frac{1}{\e}\int_{v_i(x')}^{v_i(x')+\e\vphi(x')} g(x',x_N)dx_N\underset{\e\to 0} \to g(x',v_i(x')+0),\\
\frac{w(x',v_i(x')+\e\vphi(x'))-w(x',v_i(x'))}{\e}\underset{\e\to 0}\to\partial_{x_N}w(x',v_i(x')+0).\\
\end{align*}
Thus, we find that, for any $\vphi\in C^\infty_c(B')$ such that $\vphi\geq 0$,
\begin{align}
{I'(v_i)}^+\cdot\vphi=\int_{B'}w(x',v_i(x'))\bigg({\frac{\G_{x'} v_i(x')}{\sqrt{1+|\G_{x'} v_i(x')|^2}}}\bigg)\cdot\G_{x'}\vphi(x')dx'\notag\\
+\int_{B'}(\partial_{x_N}w(x',v_i(x')+0)\sqrt{1+|\G_{x'} v_i(x')|^2}-(t_i-g(x',v_i(x')+0)))\vphi(x')dx'.
\end{align}

Our aim is now to integrate by parts in the first integral that we shall denote $\tilde{I}(v_i)$. For this purpose, let us also denote $f_i(x')=w(x',v_i(x'))$. It is readily checked that $f_i\in \Lip(B')\subset H^1(B')$.
Therefore, $v_i$ being regular
\begin{align}\label{chapTV:calcul_EL_w}
\tilde{I}(v_i)&=\int_{B'} f_i\bigg({\frac{\G_{x'} v_i}{\sqrt{1+|\G_{x'} v_i|^2}}}\bigg)\G_{x'}\vphi\notag\\
&=-\int_{B'}\diverg_{x'}\left(f_i\bigg({\frac{\G_{x'} v_i}{\sqrt{1+|\G_{x'} v_i|^2}}}\bigg)\right)\vphi\notag\\
&=-\int_{B'}\G_{x'} f_i \bigg({\frac{\G_{x'} v_i}{\sqrt{1+|\G_{x'} v_i|^2}}}\bigg)\vphi-\int_{B'} w(\cdot,v_i)\kappa_i\vphi
\end{align}
where we denoted $\kappa_i(x')=\diverg_{x'} \bigg({\frac{\G_{x'} v_i(x')}{\sqrt{1+|\G_{x'} v_i(x')|^2}}}\bigg)$ the mean curvature of the level set $\partial^* E_{t_i}$ at $({x}',v_i(x'))$.

Note that $\H^{N-1}$-a.e. point in $B'$ is a Lebesgue point for $\G_{x'}f_i$ so $v_i$ satisfies
\begin{align}\label{chapTV:ineq1}
-\G_{x'}f_i(x')\cdot\frac{\G_{x'} v_i(x')}{\sqrt{1+|\G_{x'} v_i(x')|^2}}
-w(x',v_i(x'))\kappa_i(x')\notag\\
+\partial_{x_N}w(x',v_i(x')+0)\sqrt{1+|\G_{{x}'} v_i(x')|^2}
-\big(t_i-g(x',v_i(x')+0)\big)\geq 0.
\end{align}
If one chooses $\e<0$ in (\ref{chapTV:deriv_I}) then one obtains in the same way 
\begin{align}\label{chapTV:ineq2}
-\G_{x'}f_i(x')\cdot\frac{\G_{x'} v_i(x')}{\sqrt{1+|\G_{x'} v_i(x')|^2}}
-w(x',v_i(x'))\kappa_i(x')\notag\\
+\partial_{x_N}w(x',v_i(x')-0)\sqrt{1+|\G_{{x}'} v_i(x')|^2}
-\big(t_i-g(x',v_i(x')-0)\big)\leq 0.
\end{align}


These identities hold for a.e. ${x}'\in B'$. Since we assumed that the contact set $\{v_1=v_2\}$ has positive $\H^{N-1}$-measure, then we can find a contact point $x'\in B'$ that satisfies the inequalities (\ref{chapTV:ineq1}) and (\ref{chapTV:ineq2}), identities (\ref{chapTV:inclusion_courb}) and such that $\G_{x'}f_1(x')=\G_{x'} f_2(x')$ (indeed $f_1=f_2$ on the contact set) which implies that
\begin{align}\label{chapTV:calc_born_courbure}
\partial_{x_N}w(x',x_N-0)\sqrt{1+|\G_{{x}'} v_1(x')|^2}-(t_1-g(x',x_N-0))
\leq w(x)\kappa_1(x')\notag\\
\leq w(x)\kappa_2(x')
\leq\partial_{x_N}w(x',x_N+0)\sqrt{1+|\G_{{x}'} v_2(x')|^2}-(t_2-g(x',x_N+0)).
\end{align}
It follows that for $\H^{N-1}$-every $x'\in B'$
\begin{align*}
0<t_2-t_1\leq\big(\partial_{x_N}w(x',x_N+0)-\partial_{x_N}w(x',x_N-0)\big)(1+\eta)\\
+\big(g(x',x_N+0)-g(x',x_N-0)\big)
\end{align*}
with $\eta$ that can be chosen as small as one wishes by taking a smaller ball $B'$.
Thus $\partial_{x_N}w$ or $g$ jumps at $\bar{x}$ hence (\ref{chapTV:jump_inclusion}).
Moreover, from the previous inequality one has for the value of the jump
\begin{align}\label{chapTV:calc_jump_decrease}
(u^+-u^-)(x)\leq(\partial_{x_N}w^+-\partial_{x_N}w^-)(x)+(g^+-g^-)(x)
\end{align}
which furnishes (\ref{chapTV:cont_decrease}).
The claim on the mean curvature follows at once from (\ref{chapTV:calc_born_courbure}). 
\end{proof}
\newpage
\begin{rmq}
\begin{trivlist}
\item[$(i)$]Assume that the discontinuity of $\partial_{x_N} w$ occurs in the opposite direction of that of $g$, namely
\begin{align*}
(\partial_{x_N}w^+-\partial_{x_N}w^-)(x)+(g^+-g^-)(x)=0.
\end{align*}
Then one can simply erase the jump of $g$: indeed from (\ref{chapTV:calc_jump_decrease}) one sees that $u$ has no discontinuity at $x$.
\item[$(ii)$]Note that if one is merely interested in the jump inclusion (\ref{chapTV:jump_inclusion}), it can be obtained by copying and pasting the proof given in the anisotropic setting: indeed reasoning by contradiction one can assume that $\H^{N-1}\big(J_u\setminus (J_g\cup J_{\G w})\big)>0$ and the rest follows.
\end{trivlist}
\end{rmq}


\section{An open question}
In this section, we assume that the weight is constant \textit{i.e.} $w=\la>0$ which corresponds to ROF's model. 
In this case, a natural and interesting question is to understand whether given two regularization parameters $\la,\ \mu>0$, $J_{u_\mu}\subset J_{u_\la}$. 
From Theorem \ref{chapTV:thm_jump_f} and Theorem \ref{chapTV:thm_jump_w}, one can get a similar inclusion principle for the discontinuity set of the solution of the TV flow (see \cite{ChamJump,JalalzaiFlow}). More precisely, if the initial datum $g\in L^N(\O)$ and $t>t'>0$ then $J_{u(t)}\subset J_{u(t')}$. If in addition $g\in \BV\cap L^N(\O)$ and $t>t'\geq0$, then $J_{u(t)}\subset J_{u(t')}\subset J_g$. Using the latter result, one can solve our question in the 1D case \cite{Briani}, in the radial case \cite{JalalzaiStairc,JalalzaiPhD} and also when $g=\chi_C$ the characteristic of a convex set \cite{Alter,Andreu}.\\

We are going to state some partial results that suggest that this inclusion principle holds for ROF's model in full generality. Before getting further we need the following lemma:

%
\begin{lem}\label{chapTV:jump_lambda}
Consider an open set $\O$ with finite Lebesgue measure, let $g\in L^\infty(\O)$ and consider respectively two minimizers $u_\la, u_{\mu}$ of (\ref{chapTV:ROF}) corresponding to the regularization parameters $\la,\ \mu>0$
then
\begin{align*}
{\|u_\la-u_{\mu}\|}_\infty\leq\frac{2|\O|{\|g\|}_{\infty}}{\min{(\la,\mu)}}|\la-\mu|.
\end{align*}
\end{lem}

\begin{proof}
Without loss of generality, one can assume that $\mu>\la$. The minimizers $u_\la$ and $u_{\mu}$ satisfy the Euler-Lagrange equation for ROF \textit{i.e.} there exist $z_\la,z_{\mu}\in L^\infty(\O,\R^N)$ such that
\begin{align*}
\begin{cases}
-\la\diverg z_\la+u_\la=g,\\
-\mu\diverg z_{\mu}+u_{\mu}=g.
\end{cases}
\end{align*}
Multiplying the first equation by $\mu/\la$, the second by $-1$ and summing the resulting identities we get
\begin{align*}
\langle-\mu\diverg(z_\la-z_{\mu})+u_\la-u_{\mu},\vphi\rangle=\left(\frac{\mu}{\la}-1\right)\langle(g-u_\la),\vphi\rangle.
\end{align*}
for any test function $\vphi\in L^2(\O)$.
If for some \textit{even} integer $p\geq2$ we set $\vphi={(u_\la-u_{\mu})}^{p-1}$ and denote $q=p/(p-1)$ the adjoint of $p$, it follows
\begin{align*}
\mu\langle(z_\la-z_{\mu}),(p-1){(u_\la-u_{\mu})}^{p-2} D(u_\la-u_{\mu})\rangle+{\|u_\la-u_{\mu}\|}_p^p\\
\leq\left(\frac{\mu}{\la}-1\right){\|g-u_\la\|}_q{\|u_\la-u_{\mu}\|}_p^{p-1}.
\end{align*}
Though, the first term on the left side of the inequality is non-negative since $\partial TV$  is a monotone operator (see \cite{Brezis73}) so we are simply left with
\begin{align}\label{chapTV:der_u_la}
{\|u_\la-u_{\mu}\|}_p\leq\frac{\mu-\la}{\la}{\|g-u_\la\|}_q
\end{align}
which implies
\begin{align}\label{chapTV:der_u_la_2}
{\|u_\la-u_{\mu}\|}_p&\leq|\O|^{\frac{1}{q}}\frac{\mu-\la}{\la}{\|g-u_\la\|}_\infty\notag\\
&\leq2|\O|^{\frac{1}{q}}{\|g\|}_\infty\frac{\mu-\la}{\la}
\end{align}
which yields the result making $p\to+\infty$.
\end{proof}
\begin{rmq}
\begin{trivlist}
\item [$(i)$]Equation (\ref{chapTV:der_u_la}) in conjunction with (\ref{chapTV:ineg_g_BV}) implies that for $g\in\BV$,
\begin{align*}
{\|u_\la-u_{\mu}\|}_2\leq\frac{|\mu-\la|}{\min(\la,\mu)}{\|g-u_\la\|}_2\leq\sqrt{2}\frac{\mu-\la}{\sqrt{\min(\la,{\mu})}}\left(\int_\O|Dg|\right)^{\frac{1}{2}},
\end{align*}
where the rightmost bound does not depend on $|\O|$ hence we can relax the assumption on $\O$.
\item [$(ii)$] Inequality (\ref{chapTV:der_u_la_2}) suggests some differentiability property for the mapping 
\begin{align*}
\begin{cases}
\R^+&\to L^p(\O)\\
\la&\mapsto u(\la)
\end{cases}
\end{align*}
defined for $p\in[2,+\infty]$. Unfortunately Rademacher's theorem fails in the infinite dimensional setting. Nonetheless, in our problem we can actually get Fréchet-differentiability almost everywhere from \cite[Corollary 5.21]{Benyamini} whenever the destination space has the so-called Radon-Nikodym Property (RNP). A space satisfies the RNP whenever it is a separable dual Banach space or a reflexive space (see \cite[Corollary 5.12]{Benyamini} ). This is indeed true for any $L^p(\O)$ space with $p<\infty$ but not for $L^\infty(\O)$ and furnishes the differentiability for the $\|\cdot\|_p$ norm only.

Note that in general it is not trivial to get Fréchet-differentiability for a generic mapping with values in a space of infinite dimension. The only positive answer in this direction states that every real-valued Lipschitz function on an Asplund space has points of Fréchet differentiability (see \cite{Preiss90} but also \cite{Bogachev96} for counterexamples). In general, the result does not even hold after convolution (see for instance \cite{Bogachev07}).
Though the situation for Gâteaux is more favorable:
the idea is that every Lipschitz map from a separable Banach
space into a space with the RNP is Gâteaux differentiable almost everywhere in the sense of Aronszajn (\cite[Proposition 6.41 and Theorem 6.42]{Benyamini} ).\\
\end{trivlist}
\end{rmq}
Lemma \ref{chapTV:jump_lambda} helps us prove the following result that says essentially that the highest jumps form a decreasing sequence with respect to the regularization parameter $\la$:\\
\begin{prop}
Let an open domain $\O\subset\R^N$ of finite Lebesgue measure, $g\in L^\infty(\O)$ non identically zero and $\la,\mu$ positive such that for some real $\e>0$
\begin{align*}
|\mu-\la|\leq\frac{\e\min{(\la,\mu)}}{2|\O|{\|g\|}_\infty}.
\end{align*}
Let also $u_\la$ and $u_{\mu}$ be two minimizers of (\ref{chapTV:ROF}).
Then if  we denote
\begin{align*}
J_{u_\la}^\e:=\{x\in J_{u_\la}\ /\ (u_\la^+-u_\la^-)(x)>\e\},
\end{align*}
one has
\begin{align*}
J_{u_\la}^\e\subset J_{u_{\mu}}
\end{align*}
up to a $\H^{N-1}$-negligible set.
\end{prop}
\begin{proof}
This proposition is a straightforward application of the preceding lemma which implies 
${\|u_{\la}-u_{\mu}\|}_\infty\leq\e.$
Then clearly for $\H^{N-1}$-almost any $x\in J_{u_{\la}}^\e$
\begin{align*}
\e<(u_\la^+-u_\la^-)(x)\leq\e+(u_{\mu}^+-u_{\mu}^-)(x)
\end{align*}
hence the conclusion.
\end{proof}

\section{Conclusion and perspective}
In this paper, we examined some fine results for energies involving terms that behave like the total variation. In particular, we prove that no new discontinuities are created for energies involving a smooth elliptic anisotropy and a generic fidelity term. This extends the result of \cite{ChamJump} where they dealt with the denoising problem. On the other hand, we characterized creation of unobserved discontinuities for the adaptive total variation functional if the weight is merely Lipschitz continuous. In addition, we proved that the infinity norm is decreased at the discontinuity while minimizing ROF's energy, which is quite counterintuitive.\\

The aforementioned results motivate many interesting questions that remain unsettled and pave the way for future researches. First of all, most of the results of this paper rely heavily on the connection with the perimeter problem via the coarea formula and it does not seem clear to us how they can be adapted to take into account linear perturbations of the data (convolution but also Radon or Fourier transforms).\\

Concerning the problem of inclusion of the discontinuities, it is not clear whether the discontinuities form a monotone sequence for a general datum. Indeed, in this case the connection with the flow fails. This question seems to be related to the existence of a smooth underlying calibration $z$ (obviously not $C^1$) for the ROF problem. This question is actually interesting by itself. But ``Finding a calibration remains an art, not a science'' as would say Frank Morgan. Here, we should also mention the work of Bourgain-Brezis \cite{Bourgain} and De Pauw-Pfeffer \cite{DePauw} where the authors were interested in finding a continuous $z$ such that
\begin{align*}
\diverg(z)=\mu
\end{align*}
for a given Radon measure $\mu$. Though these results are not constructive since referring to the axiom of choice and cannot be easily adapted. The inclusion could also be obtained by establishing strong properties of the derivative $u'(\la)$ by means of $\Gamma$-convergence for instance (see \cite{Braides}). Though the resulting functional seems to be non-local making the problem difficult (see \cite{Camar}).\\

Another problem, which seems within reach, would be to examine the regularity of the minimizer for the general energy we considered. This could be done by adapting \cite{ChamRegularity} where such a result is established for the denoising problem.\\


\bibliographystyle{siam}
\bibliography{biblio}{}
\end{document}